\documentclass[a4paper,11pt]{article}

\usepackage{amsmath,amssymb,amsthm,amscd}
\usepackage[mathscr]{eucal}
\usepackage[all]{xy}
\usepackage{color}
\usepackage{graphicx}

\makeatletter
    
    \@addtoreset{equation}{section}
  \makeatother

\newcommand{\plim}{\varprojlim}
\newcommand{\mcal}{\mathcal}
\newcommand{\mbf}{\mathbf}

\newcommand{\mfrak}{\mathfrak}
\newcommand{\mbb}{\mathbb}
\newcommand{\mrm}{\mathrm}
\newcommand{\vphi}{\varphi}

\newcommand{\aet}{\mathrm{\acute{e}t}}
\newcommand{\cO}{\mathcal{O}}

\newtheorem{theorem}{Theorem}[section]
\newtheorem{corollary}[theorem]{Corollary}
\newtheorem{lemma}[theorem]{Lemma}
\newtheorem{proposition}[theorem]{Proposition}

\theoremstyle{definition}
\newtheorem{definition}[theorem]{Definition}
\newtheorem{remark}[theorem]{Remark}
\newtheorem*{claim}{Claim}

\newtheorem*{question}{Question}

\newtheorem*{acknowledgments}{Acknowledgments}

\title{Kummer-faithfulness over $p$-adic fields}

\author{Yoshiyasu Ozeki\footnote{
Faculty of Science, Kanagawa University,
3-27-1 Rokkakubashi, Kanagawa-ku, Yokohama-shi, Kanagawa 221-8686, JAPAN
\endgraf
e-mail: {\tt ozeki@kanagawa-u.ac.jp}} 
}

\begin{document}
\maketitle

\begin{abstract}
The notion of a Kummer-faithful field, 
defined by Mochizuki, is expected as one of suitable base fields for anabelian geometry. 
In this paper, we study Kummer-faithfulness  
for algebraic extension fields of $p$-adic fields. 
We show that Kummer-faithfulness for such fields are deeply related with various finiteness properties on 
torsion points of (semi-)abelian varieties. For example,
a Galois extension $K$ of a $p$-adic field is Kummer-faithful with finite residue field 
if and only if, for any finite extension $L$ of $K$ and any abelian variety over $L$,
its $L$-rational torsion subgroup is finite.
In addition, we study Kummer-faithfulness for Lubin-Tate extension fields.
\end{abstract}

      \tableofcontents


\section{Introduction}
Anabelian geometry is an area of arithmetic geometry that studies 
how much information concerning the geometry of certain geometric objects, 
so-called ``anabelian varieties", 
can be reconstructed from various data associated with their 
arithmetic fundamental groups.
The philosophy of anabelian geometry was first suggested 
by Grothendieck in {\it Esquisse d'un Programme} and 
{\it Brief an G. Faltings} (cf. \cite{ScLo97}), and 
he proposed that anabelian geometry
should be considered over fields that are 
finitely generated over their prime fields.
Nowadays, Grothendieck's original conjecture 
for hyperbolic curves over finitely generated fields over 
the prime field has been proved by 
Nakamura (\cite{Nak90a}, \cite{Nak90b}), 
Tamagawa (\cite{Tam97}), Mochizuki (\cite{Moc96}, \cite{Moc07})
and Stix (\cite{Sti02a}, \cite{Sti02b}).

However, led by Mochizuki, 
it has been revealed that anabelian geometry 
can be developed over a broader class of fields.
{\it Kummer-faithful fields}, 
defined by Mochizuki \cite{Moc15} and the main topic of this paper,
form one such class. 
A perfect field $K$ is Kummer-faithful if,
for every finite extension $L$ of $K$ and every semi-abelian variety $A$ over $L$,
the Mordell-Weil group $A(L)$ has a trivial divisible part;
see Definition \ref{Def:KF}.
(Precisely, in \cite{Moc15}, 
Kummer-faithful fields are assumed to be of characteristic zero. 
In \cite{Hos17}, Kummer-faithful fields are defined also for 
positive characteristic cases.) 
Kummer-faithfulness asserts the injectivity of the Kummer map associated with semi-abelian varieties; 
thus, roughly speaking, Kummer-faithfulness guarantees that 
``Kummer theory for semi-abelian varieties works effectively''.
Typical examples of Kummer-faithful fields are 
finitely generated fields over $\mbb{Q}$,
which are related to the Grothendieck's original setting.
Moreover, {\it sub-$p$-adic fields} are Kummer-faithful
(here, a field is sub-$p$-adic
if it is isomorphic to a finitely generated field over 
$\mbb{Q}_p$).
As one of the important results related to the Grothendieck conjecture 
over Kummer-faithful fields, 
Hoshi proved in \cite{Hos17} that a ``point-theoretic'' and 
``Galois-preserving'' isomorphism  between 
the \'etale fundamental groups of affine hyperbolic 
curves over Kummer-faithful fields arises 
from an isomorphism of schemes.

A study of various properties of Kummer-faifthfulness, together with 
the construction of examples of Kummer-faithful fields,  
is important for understanding the range of fields 
over which anabelian geometry can be developed. 
Moreover, in recent years, studying these topics  has become a research interest
in its own right.
Let $k$ be a number field and $K/k$ an algebraic extension.
In \cite[Corollary 2.15]{OzTg22}, Taguchi and the author studied  Kummer-faithfulness of $K$
in terms of ramification theory.
We showed that 
$K$ is Kummer-faithful 
if $K/k$ is a  Galois extension with ``finite maximal ramification break everywhere'' (cf.\ Definition 2.14 of 
{\it loc.\ cit.}).
As a typical example, the field 
obtained by adjoining to $\mbb{Q}$ all $\ell$-th roots of unity for all prime $\ell$ is Kummer-faithful.
Ohtani \cite{Oht22} and Asayama-Taguchi \cite{AsTg24} 
studied Kummer-faithfulness 
for extremely large fields. 
Let $G$ be the absolute Galois group of $k$ and  
$e$ a positive integer. 
For $\sigma =(\sigma_1,\dots, \sigma_e) \in G^e$
denote by $\overline{k}(\sigma)$  
the fixed field of $\sigma$ in $\overline{k}$, 
and by $\overline{k}[\sigma]$ the maximal Galois subextension 
of $k$ in $\overline{k}(\sigma)$. 
It is known that
$\overline{k}[\sigma]$ is
Kummer-faithful for almost all $\sigma \in G^e$
(in terms of the (normalized) Haar measure); 
see \cite[Corollary 1]{Oht22} for the case where $e\ge 2$, and \cite[Theorem 5.3]{AsTg24} for the general case.
Moreover, if $e\ge 2$, 
$\overline{k}(\sigma)$ is Kummer-faithful for almost all $\sigma \in G^e$ (cf. \cite[Theorem 5.2]{AsTg24}).
The structure of the Mordell-Weil group of (semi-)abelian varieties over $\overline{k}(\sigma)$ or $\overline{k}[\sigma]$ satisfies
interesting properties; see Section 3 of \cite{AsTg24}
for more information.
On the other hand, 
Murotani showed in \cite{Mur23b} that an algebraic extension field $\mbb{F}$ over $\mbb{F}_p$ is Kummer-faithful if and only if 
the absolute Galois group of $\mbb{F}$ is isomorphic to 
$\hat{\mbb{Z}}$.
It should be notable that he also studied Kummer-faithfulness for higher local fields
(cf.\ \cite{Mur23a}).

In this paper, we study Kummer-faithfulness for   
algebraic extension fields of
$p$-adic fields (= finite extension fields of $\mbb{Q}_p$).
Since sub-$p$-adic fields are Kummer-faithful,  
we know that $p$-adic fields are Kummer-faithful
(but this can be checked immediately from the main theorem of \cite{Mat55}).
If we restrict our attention to tamely ramified (or unramified)
Galois extensions $K$ of some $p$-adic field, 
we will see that Kummer-faithfulness has a simple interpretation; 
in fact, for such $K$, 
$K$ is Kummer-faithful if and only if $K/\mbb{Q}_p$
is quasi-finite (see Corollaries \ref{unram} and \ref{tame}).
Thus our main interest is Kummer-faithfulness for
algebraic extensions of $p$-adic fields with infinite wild ramification.
For this, we focus on Lubin-Tate
extension fields of $p$-adic fields.
For a power $q$  of $p$, 
we say that $\alpha$ is a {\it Weil $q$-integer} if
$|\iota(\alpha)|=\sqrt{q}$ 
for every embedding $\iota\colon \mbb{Q}(\alpha)\hookrightarrow \mbb{C}$. 
We denote by $W(q)$ the set of Weil $q$-integers.

\begin{theorem}[A part of Theorem \ref{refined:LT}]
\label{LT}
Let $k$ be  a $p$-adic field with residue field $\mbb{F}_q$ 
and $\pi$ a uniformizer of $k$.
Denote by $k_{\pi}$ the Lubin-Tate extension field of $k$ associated with $\pi$.
\begin{itemize}
\item[{\rm (1)}]  
If $k$ is a Galois extension of $\mbb{Q}_p$ and $k_{\pi}$ is not Kummer-faithful, 
then either of the following holds.
\begin{itemize}
\item[{\rm (a)}] $q^{-1}\mrm{Nr}_{k/\mbb{Q}_p}(\pi)\in \mu_{p-1}$.
\item[{\rm (b)}] For some $(r_{\sigma})_{\sigma \in \Gamma_k}$
with $r_{\sigma}\in \{0,1\}$, 
it holds $\prod_{\sigma \in \Gamma_k} \sigma \pi^{r_{\sigma}}\in W(q)$.
Here,  $\Gamma_k$ is the set of $\mbb{Q}_p$-algebra embeddings $k\hookrightarrow \overline{\mbb{Q}}_p$.
\end{itemize}

\item[{\rm (2)}] 
If  $q^{-1}\mrm{Nr}_{k/\mbb{Q}_p}(\pi)\in \mu_{p-1}$ or 
$\mrm{Nr}_{k/\mbb{Q}_p}(\pi)\in   W(q)$,
then $k_{\pi}$ is not Kummer-faithful.
\end{itemize}
\end{theorem}
Applying Theorem \ref{LT} with $k=\mbb{Q}_p$, 
we obtain 
\begin{corollary}
\label{LT:k=Qp}
Assume $k=\mbb{Q}_p$. Then the following are equivalent.
\begin{itemize}
\item[{\rm (i)}] $k_{\pi}$ is Kummer-faithful.
\item[{\rm (ii)}]$p^{-1}\pi\notin \mu_{p-1}$ and $\pi\not \in W(p)$.
\end{itemize}
\end{corollary}

Similar results related to assertion (1) of Theorem \ref{LT} have already been studied in 
\cite[Theorem 1.1]{Oze23}, 
and we will essentially follow the same arguments provided there in our proof.
Assertion (2), which can be seen as 
a partial converse of (1), 
is not studied in {\it loc.\ cit}.
In Section \ref{nonKE:example},
we construct an example of non-Kummer faithful 
$k_{\pi}$ such that the assumptions in Theorem \ref{LT} (2) do not hold but (b) in Theorem \ref{LT} (1) holds. 
The author has a slight hope that  $k_{\pi}$ is not Kummer-faithful
if and only if either (a) or (b) in Theorem \ref{LT} (1) 
holds.

On the other hand,
in the $p$-adic setting, 
thanks to the theory of Tate curves and non-archimedian rigid uniformization theorems, we can show that  
Kummer-faithfulness is equivalent to 
certain  finiteness properties of torsion points of abelian varieties.
\begin{theorem}[Corollary of Theorem \ref{KF:char1}]
\label{Intro:KF}
Let $K$ be a Galois extension of a $p$-adic field.
\begin{itemize}
\item[{\rm (1)}] 
The following are equivalent.
\begin{itemize}
\item[{\rm (a)}] $K$ is Kummer-faithful.
\item[{\rm (b)}] Any finite extension of $K$ has
only finitely many $\ell$-power roots of unity for every prime $\ell$, and 
the group $A(L)[p^{\infty}]$ is finite
for any finite extension $L/K$ 
and any abelian variety $A/L$ with good reduction.
\end{itemize}
\item[{\rm (2)}] 
The following are equivalent.
\begin{itemize}
\item[{\rm (a)}] $K$ is Kummer-faithful with finite residue field.
\item[{\rm (b)}] The torsion subgroup of $A(L)$ is finite
for any finite extension $L/K$ 
and any abelian variety $A/L$.
\end{itemize}
\end{itemize}
\end{theorem}
The finiteness properties such as (b) of Theorem \ref{Intro:KF} (2) have been studied as a standard problem
in arithmetic theory of abelian varieties. 
It seems that many known results are for abelian varieties over number fields, but some results are also known for those  over $p$-adic fields.
It is a theorem of Imai \cite{Ima75} that,
if $K=\mbb{Q}_p(\mu_{p^{\infty}})$, 
the torsion subgroup of $A(L)$ is finite
for any finite extension $L/K$ 
and any abelian variety $A/K$ with potential good reduction.
This result is known to be essentially valuable in studies such as Iwasawa theory.
Kubo and Taguchi \cite{KbTg13} generalized Imai's theorem in the sense that
Imai's theorem holds even after replacing $\mbb{Q}_p(\mu_{p^{\infty}})$
with the field $k(k^{1/p^{\infty}})$ 
obtained by adjoining to a $p$-adic field $k$
all $p$-power roots of all element of $k$.
However, the  fields $K$ appearing here are not Kummer-faithful
since Tate curves over $\mbb{Q}_p$ contain non-trivial divisible $K$-rational element
(see also Corollary \ref{LT:k=Qp}). 

It would be very interesting to consider analogues 
of Theorem  \ref{Intro:KF} over global fields; however, 
the author currently has no idea what such statements might look like.

\begin{acknowledgments}
The author would like to express sincere gratitude to Yuichiro Taguchi for valuable advice given in the early stages of this paper.
Special thanks are also due to Akio Tamagawa for helpful comments on the results in Section 4.
The author is furthermore grateful to Takahiro Murotani for insightful advice on the historical aspects of anabelian geometry.
\end{acknowledgments}

\vspace{5mm}
\noindent
{\bf Notation.}
A {\it number field} is a finite 
extension of the field  $\mbb{Q}$ of rational numbers.
Let $p$ be a rational prime.
A {\it $p$-adic field} is a finite extension of the field  $\mbb{Q}_p$ of $p$-adic numbers.
For any field $F$, 
we fix a separable closure $\overline{F}$ of $F$
and we  denote by $G_F$
the absolute Galois group $\mrm{Gal}(\overline{F}/F)$ of $F$.

\section{General theory for Kummer-faithful fields}

In this section, we recall the definition of  Kummer-faithful fields 
(cf.\ \cite[Def.\ 1.5]{Moc15}, \cite[Def.\ 1.2]{Hos17})
and study some standard properties for Kummer-faithfulness
in general settings. 
Before starting the main part, 
we briefly recall the degree of (not necessary finite) algebraic extensions.
A {\it supernatural number} is a formal product 
$\mfrak{n}=\prod_{\ell}\ell^{n_{\ell}}$ 
where $\ell$ runs over the set of primes 
and $n_{\ell}\in \mbb{Z}\cup \{\infty \}$.
Using the unique decomposition into prime powers,
we can view any natural number as a supernatural number.
The product of supernatural numbers are defined by the natural way,
and also the greatest common divisors and the least common multiple of 
supernatural numbers. 
Let $L$ be an algebraic extension of a perfect field $K$.
We define the {\it extension degree} $[L:K]$ of $L/K$ 
by 
\begin{equation*}
[L:K]:=\mrm{lcm}\{[K':K] \mid \mbox{$K'$ is a finite extension of  $K$ contained in  $L$}\}
\end{equation*}
as a supernatural number.
By Galois theory, we have a group theoretic interpretation of $[L:K]$ as follows.
Take any Galois extension $\tilde{L}$ of $K$ which contains $L$,
and set $G:=\mrm{Gal}(\tilde{L}/K)$ and $H:=\mrm{Gal}(\tilde{L}/L)$.
Then we see  $[L:K]=(G:H)$,
where the right hand side is the index of profinite groups in the sense of Section 1.3 of \cite{Ser97}.
For algebraic extensions $K\subset L\subset M$,
we have $[M:K]=[M:L]\cdot [L:K]$.
If an algebraic extension $L/K$ is of the form $L=\bigcup_i K_i$ for 
some finite extensions $K_i$ of $K$,
one has $[L:K]=\mrm{lcm}\{[K_i:K]\}_{i}$.
\begin{definition}
Let $L$ be an algebraic extension of a  field $K$
with $[L:K]=\prod_{\ell}\ell^{n_{\ell}}$.
We say that $L/K$ is {\it quasi-finite} if $n_{\ell}$ is finite for any prime $\ell$.
\end{definition}

Every finite extensions are clearly quasi-finite.
For an algebraic extension $\mbb{E}$  of a finite field $\mbb{F}$ with 
$[\mbb{E}:\mbb{F}]=\prod_{\ell}\ell^{n_{\ell}}$,
then one sees that
$\mbb{E}/\mbb{F}$ is quasi-finite if and only if 
$G_{\mbb{E}}$ is isomorphic to $\hat{\mbb{Z}}$.
If  $L$ is an unramified extension of a $p$-adic field  $K$,
then $L/K$ is quasi-finite 
if and only if the residue field extension of  $L/K$ is quasi-finite.

Now let us recall the definition of Kummer-faithful fields
and study some basic properties.
Let $M$ be a $\mbb{Z}$-module and $\ell$ a prime. 
We say that 
$P\in M$ is {\it divisible} (resp.\ {\it $\ell$-divisible}) 
if, for any integer $n>0$, 
there exists $Q\in M$ such that $P=nQ$ (resp.\ $P=\ell^nQ$).  
We denote by $M_{\mrm{div}}$ (resp.\ $ M_{\ell\text{\rm{-div}}}$)
the set of divisible (resp.\ $\ell$-divisible) elements of $M$, that is,
$$
M_{\mrm{div}}=\bigcap_{n>0} n M,     \qquad
M_{\ell\text{\rm{-div}}}=\bigcap_{n>0}\ell^n M.
$$ 
Note that 
$P\in M$  is divisible if and only if it is $\ell$-divisible for all primes $\ell$.
In fact, we have
$M_{\mrm{div}}=\bigcap_{\ell} M_{\ell\text{\rm{-div}}}$.
If a $\mbb{Z}$-module $M$ is divisible (i.e., $M_{\mrm{div}}=M$),
it is known  (e.g., \cite[Chapter 4, Theorem 3.1]{Fuc15}) that  $M$ is isomorphic to 
$\mbb{Q}^{I}\oplus \left(\oplus_{p} \left(\mbb{Z}[1/p]/\mbb{Z}\right)^{I_p}\right)$,
where $p$ runs over the set of primes and $I$, $I_p$ are some index sets.

\begin{definition} 
\label{Def:KF}
A perfect field $K$ is {\it Kummer-faithful} 
(resp.\  {\it AVKF}, resp.\  {\it torally Kummer-faithful})
if, for every finite extension $L$ of $K$ and every semi-abelian variety 
(resp.\ abelian variety,\ resp.\ torus)
$A$ over $L$, it holds that 
$$
A(L)_{\rm div} =0.
$$ 
\end{definition}
It is clear that  any subfield of a Kummer-faithful field is also Kummer-faithful,
and any finite extension of a Kummer-faithful field is also Kummer-faithful.
As is remarked in \cite[Rem.\ 1.5.2]{Moc15}, 
by considering the Weil restriction, 
one verifies immediately that one obtains
an equivalent definition of Kummer-faithful, 
if, in Definition \ref{Def:KF}, one restricts $L$ to be equal to $K$.
The same statements holds if we replace ``Kummer-faithful" 
with ``AVKF" or ``torally Kummer-faithful".

It is  shown by Mochizuki in Remark 1.5.4  of \cite{Moc15} that 
any {\it sub-$p$-adic field} is Kummer-faithful. 
Here, recall that a field $k$ is sub-$p$-adic if there exists a prime  $p$  and 
a finitely generated field extension $L$ of $\mbb{Q}_p$ such that 
$k$ is isomorphic to a subfield of $L$. 
Let $F^{\mrm{cyc}}$ be the maximal cyclotomic field of a number field $F$. 
Then it follows from the result of Ribet (and Proposition \ref{div})
that $F^{\mrm{cyc}}$ is AVKF (however, it is not Kummer-faithful).
Furthermore, if $L$ is a quasi-finite Galois extension of $F^{\mrm{cyc}}$,
then $L$ is also AVKF (cf.\ \cite[Proposition 6.3 (i)]{HoMoTs20}).
Furthermore, it is shown by Murotani and the author 
in \cite[Theorem 1.3]{MrOz25} that the field 
$F(F^{1/\infty})$ obtained by adjoining to $F$
all roots of all elements of $F$ is AVKF.

\begin{proposition} 
\label{KFreduction}
A perfect field 
$K$ is Kummer-faithful if and only if
it is both torally Kummer-faithful and AVKF.
\end{proposition}

\begin{proof}
Suppose that $K$ is both torally Kummer-faithful and AVKF.
Let $L$ be a finite extension of $L$ and $A$ a semi-abelian variety over $L$.
There exists an exact sequence  $0\to T\to A\to B\to 0$ of $K$-group schemes
where $T$ is a torus and $B$ is an abelian variety,
which induces an exact sequence $0\to T(L)\to A(L)\to B(L)$
of $\mbb{Z}$-modules.
Since $K$ is both torally Kummer-faithful and AVKF, 
we know that both $T(L)_{\mrm{div}}$ and $B(L)_{\mrm{div}}$
are zero.
Hence it follows from Lemma \ref{divlem} below that $A(L)_{\mrm{div}}$ is also zero.
\end{proof}

\begin{lemma}
\label{divlem}
Let $0\to L\to M\to N$ be an exact sequence of $\mbb{Z}$-modules.
Assume that $L_{\mrm{div}}=0$, $N_{\mrm{div}}=0$ and 
the kernel of  the $n$-th multiplication map on $M$ is finite for any $n>0$.
Then we have $M_{\mrm{div}}=0$.
\end{lemma}

\begin{proof}
Take any $x\in M_{\mrm{div}}$. 
For any $n>0$, 
we denote by $X_n$ the set of all $y\in M$ such that $ny=x$.
Then $\{X_n\}_{n>0}$ forms a projective system with transition maps
$f_{n,m}\colon X_m\to X_n$ given by $f_{n,m}(z)=(m/n)z$ for $n\mid m$.
Since each $X_n$ is a non-empty finite set, 
the projective limit $\plim_{n} X_n$ is non-empty.
Take any $(x_n)_n\in \plim_{n} X_n$
and denote by $f\colon M\to N$ the map in the statement of the lemma.
Since each $x_n$ is a divisible element of $M$,
 we see $f(x_n)\in N_{\mrm{div}}$. This shows $f(x_n)=0$ and thus $x_n\in L$.
Since $x=nx_n\in L$ for any $n$, we find that  $x$ is a divisible element of $L$.
This gives $x=0$ as desired. 
\end{proof}

We study some relations between Kummer-faithfulness 
and finiteness of prime power torsion part of semi-abelian varieties.

\begin{definition} 
\label{Def:fine}
(1) Let $\ell$ be a prime.
A perfect field $K$ is {\it $\ell^{\infty}$-semi-AV-tor-finite} 
(resp.\  {\it $\ell^{\infty}$-AV-tor-finite}) 
if, for every finite extension $L$ of $K$ and every semi-abelian variety (resp.\ abelian variety) 
$A$ over $L$, it holds that 
$A(L)[\ell^{\infty}]$ is finite. 

\noindent
(2) A perfect field $K$ is {\it locally semi-AV-tor-finite} 
(resp.\  {\it locally AV-tor-finite}
) if
it is  $\ell^{\infty}$-semi-AV-tor-finite (resp.\  $\ell^{\infty}$-AV-tor-finite)
for every prime $\ell$.

\noindent
(3) A perfect field $K$ is {\it semi-AV-tor-finite} 
(resp.\  {\it AV-tor-finite}) if, 
for every finite extension $L$ of $K$ and every semi-abelian variety (resp.\ abelian variety) 
$A$ over $L$, it holds that 
$A(L)_{\mrm{tor}}$ is finite. 
\end{definition}

It is helpful to the readers to refer \cite[Section 6]{HoMoTs20}
for various properties of Kummer-faithful fields, AVKF fields, 
$\ell^{\infty}$-AV-tor-finite fields and so on. 
Note that locally semi-AV-tor-finite is equivalent to 
{\it $\mfrak{Primes}^{\infty}$-AV-tor-finite}
in the sense of \cite[Definition 6.1 (iv)]{HoMoTs20}.
Any sub-$p$-adic field is semi-AV-tor-finite by Proposition 2.9 of \cite{OzTg22}. 
The theorem of Ribet \cite{KaLa81} shows that 
the maximal cyclotomic field $F^{\mrm{cyc}}$  of a number field $F$ 
is AV-tor-finite.
Moreover, it is shown in \cite[Corollary 1.2]{MrOz25} 
that the field 
$F((\cO_F^{\times})^{1/\infty})$ obtained by adjoining 
to $F$ all roots of all units of the integer ring 
of $F$ is also AV-tor-finite.

\begin{proposition} 
\label{div}
Let $K$ be an algebraic extension of a field $k$.
Let $A$ be a semi-abelian variety {\rm (}resp.\ abelian variety,\ resp.\ torus{\rm )} over $k$.
Consider the following conditions.
\begin{itemize}
\item[{\rm (a)}] $A(K)_{\rm div}$ is zero.
\item[{\rm (b)}] $A(K)[\ell^{\infty}]$ is finite for any prime  $\ell$.
\end{itemize}
Then we have $(a) \Rightarrow (b)$. 
If $k$ is  Kummer-faithful {\rm (}resp.\  AVKF, resp.\  torally Kummer-faithful{\rm )}
and $K$ is a Galois extension of $k$, then 
we have $(a)\Leftrightarrow (b)$.  
\end{proposition}

\begin{proof}
The statement for Kummer-faithful fields follows from Proposition 2.4 of \cite{OzTg22}. 
The arguments  in {\it loc. cit.} proceed also for 
AVKF fields and torally Kummer-faithful fields.
\end{proof}
Here is an immediate consequence of the proposition above.
\begin{corollary}
\label{fine=KF}
{\rm (1)} A perfect field is  locally semi-AV-tor-finite if it is Kummer-faithful. 

\noindent
{\rm (2)} For a Galois extension field of a Kummer-faithful field, 
it is  locally semi-AV-tor-finite if and only if it is Kummer-faithful.
\end{corollary}

Note that there exists a  locally semi-AV-tor-finite field which is not Kummer-faithful.
Let $a>1$ be a natural number and take a system $(a_n)_{n>0}$
in $\overline{\mbb{Q}}$ such that $a_1=1$ and $a_{nm}^m=n$ for all $n,m>0$.
Denote by $K$ the extension field of $\mbb{Q}$
obtained by adjoining all $a_n$ for $n>0$.  
Then, $K$ is  locally semi-AV-tor-finite but is not Kummer-faithful since $K^{\times}$ 
contains a non-trivial divisible element $a$.

\begin{proposition}
\label{fine:ppl:fine}
Let $\Box \in \{semi\mbox{-}AV,AV\}$
and $\ell$ a prime. 
Let $K$ be a perfect field and  $L$ a potentially prime-to-$\ell$ extension of $K$.
If $K$ is $\ell^{\infty}$-$\Box$-tor-finite,
then $L$ is also $\ell^{\infty}$-$\Box$-tor-finite.

In particular, any quasi-finite extension of a  locally semi-AV-tor-finite field 
is also locally semi-AV-tor-finite. 
\end{proposition}

\begin{proof}
Let $L'$ be a finite extension of $L$ and $A$ a (semi-)abelian variety over $L'$.
Take any finite extension $K'/K$ contained in $L'$ such  that  
$A$ is defined over $K'$.  
Setting $L'':=L'(K'(A[\ell]))$, we know that 
$L''/K'(A[\ell])$ is potentially prime-to-$\ell$ since so is $L/K$.
Thus there exists  a finite extension  $K''$ of $K'(A[\ell])$ contained in $L''$ such that 
$L''/K''$ is prime-to-$\ell$. 
Since $K''(A[\ell^{\infty}])$ is a pro-$\ell$ extension of $K''$ but 
$L''$ is a prime-to-$\ell$ extension of $K''$,
we see that the intersection $L''\cap K''(A[\ell^{\infty}])$ is equal to $K''$.
Hence we obtain 
$A(L')[\ell^{\infty}]\subset A(L'')[\ell^{\infty}]
=A(K'')[\ell^{\infty}]$.
Since $K$ is $\ell^{\infty}$-semi-AV-tor-finite,
the finiteness of  $A(K'')[\ell^{\infty}]$ is assured,  
which gives the fact that  $A(L')[\ell^{\infty}]$ is finite.
\end{proof}

It may be helpful to write down the following implications:
$$
\displaystyle \xymatrix{
\mbox{semi-AV-tor-finite}
\ar@{=>}[r] 
&
\mbox{locally semi-AV-tor-finite}
\ar@{=>}[r] 
& \mbox{$\ell^{\infty}$-semi-AV-tor-finite}
\\
\mbox{sub-$p$-adic}
\ar@{=>}[u] 
\ar@{=>}[r] 
& 
\mbox{KF = TKF and AVKF} 
\ar@{=>}^{(\ast)}[u] 
& 
}
$$
Here, KF and TKF stand for Kummer-faithful and torally Kummer-faithful, respectively. 
The vertical arrow ($\ast$) is an equivalence relation for the class of Galois extension fields 
of Kummer-faithful fields.
It follows from Theorem \ref{KF:char1} (2)  below that,  
for Galois extension fields of $p$-adic fields,
we have an equivalence 
``semi-AV-tor-finite $\Leftrightarrow$ KF with finite residue fields".

\section{Kummer-faithfulness in $p$-adic settings}

In this section, we study Kummer-faithfulness
for algebraic extensions $K$ of $\mbb{Q}_p$.
Throughout this section,
we denote by $\mbb{F}_K$ the residue field of  $K$.
First we should point out that, by the existence of Tate curves, 
one can check 
that Kummer-faithfulness is equivalent to 
AVKF property in this situation. 
\begin{proposition}
\label{KF=AVKF}
Let $K$ be an algebraic extension of $\mbb{Q}_p$.

\begin{itemize}
\item[{\rm (1)}] 
For a prime $\ell$, 
$K$ is $\ell^{\infty}$-semi-AV-tor-finite 
if and only if it is $\ell^{\infty}$-AV-tor-finite.
\item[{\rm (2)}] 
$K$ is Kummer-faithful if and only if $K$ is AVKF.
\end{itemize}
\end{proposition}
\begin{proof}
Let $E_{/\mbb{Q}_p}$ be the Tate curve associated with uniformizing element $p$.
To show  (1), it suffices to prove that 
$\mu_{\ell^{\infty}}(L)$ is finite for any finite extension $L$ of $K$
under the assumption that $K$ is $\ell^{\infty}$-AV-tor-finite, 
but  this follows immediately from the existence of an injection
$\mu_{\ell^{\infty}}(L) \hookrightarrow E(L)[\ell^{\infty}]$.
To show  (2), 
it suffices to prove that the divisible part $(L^{\times})_{\mrm{div}}$  of $L^{\times}$ is trivial
for any finite extension $L$ of $K$
under the assumption that $K$ is AVKF, 
but  this also follows immediately from Lemma \ref{divlem} and
an exact sequence $0\to p^{\mbb{Z}}\to L^{\times}\to E(L)$ of abelian groups.
\end{proof}

\subsection{Unramified extensions and tamely ramified extensions}

In this section, we give criteria of Kummer-faithfulness
for unramified, or tamely ramified, extension fields  of some $p$-adic fields.
In addition, 
we show that Kummer-faithful fields that are Galois extensions of $p$-adic fields
admit a decomposition of quasi-finite extension fields and 
(possibly of infinite degree) $p$-power extensions.

Recall that 
$\mbb{F}_K$ is the residue field of an algebraic extension field $K$ of $\mbb{Q}_p$.
We say that a field $K$ is {\it stably $\mu_{\ell^{\infty}}$-finite} if
$\mu_{\ell^{\infty}}(L)$ is finite for any finite extension $L$ of $K$.

\begin{lemma}
\label{residue}
Let $K$ be an algebraic extension of $\mbb{Q}_p$.
Denote by $G^{\ell}_{\mbb{F}_K}$ 
the maximal pro-$\ell$ quotient of $G_{\mbb{F}_K}$
for any prime $\ell$.
\begin{itemize}
\item[{\rm (1)}] 
Assume $\ell\not=p$. Then 
$K$ is stably $\mu_{\ell^{\infty}}$-finite if and only if 
$G^{\ell}_{\mbb{F}_K}\simeq \mbb{Z}_{\ell}$.
\item[{\rm (2)}]
If $K$ is $p^{\infty}$-semi-AV-tor-finite, then 
$G^{p}_{\mbb{F}_K}\simeq \mbb{Z}_{p}$. 
\end{itemize}
\end{lemma}

\begin{proof}
(1) We may suppose that $K$ contains $\mu_{\ell}$. 
Then the maximal pro-$\ell$ extension of $\mbb{F}_{K}$ is 
$\mbb{F}_{K}(\mu_{\ell^{\infty}})$, and 
$G^{\ell}_{\mbb{F}_K}
\simeq \mrm{Gal}(\mbb{F}_p(\mu_{\ell^{\infty}})/\mbb{F}_{K}\cap \mbb{F}_p(\mu_{\ell^{\infty}}))
\subset \mrm{Gal}(\mbb{F}_p(\mu_{\ell^{\infty}})/\mbb{F}_p(\mu_{\ell}))
\simeq \mbb{Z}_{\ell}$. 
Thus we have the following equivalent relations: 
$G^{\ell}_{\mbb{F}_K}\simeq \mbb{Z}_{\ell}$ 
$\Leftrightarrow $
$\# \mu_{\ell^{\infty}}(\mbb{F}_K)<\infty$
$\Leftrightarrow $
$\# \mu_{\ell^{\infty}}(K)<\infty$.
The result immediately follows.

(2) 
Assume  $G^{p}_{\mbb{F}_K}\not\simeq \mbb{Z}_{p}$, that is, $G^{p}_{\mbb{F}_K}$ is trivial. 
Take a CM elliptic curve  $E$ defined over a $p$-adic subfield $k$ of $K$
with the properties that 
$E$ has good ordinary  reduction over $k$ and 
every endomorphism of $E$ is defined over $k$. 
Let  $V_p(E)$ and $V_p(\bar{E})$ be the $p$-adic Tate module
of $E$ and that of the reduction $\bar{E}$ of $E$,
respectively.
With a suitable choice of a basis of $V_p(E)$,
the representation $\rho\colon G_k\to GL_{\mbb{Q}_p}(V_p(E))\simeq GL_2(\mbb{Q}_p)$
is of the form 
$$
\rho= 
\begin{pmatrix}
\chi_p \varepsilon^{-1} & u\\
0 & \varepsilon
\end{pmatrix},
$$
where $\chi_p$ is the $p$-adic cyclotomic character, $\varepsilon$
is an unramified character defined by the $G_k$-action on $V_p(\bar{E})$ 
 and $u$ is a continuous map. 
Since  $G^{p}_{\mbb{F}_K}$ is trivial, we find that  
an open subgroup of $G_{\mbb{F}_K}$ acts on $\bar{E}[p^{\infty}]$ trivial.
Replacing $K$ by a finite extension, 
we may assume that  $\bar{E}[p^{\infty}]$ is defined over  $\mbb{F}_K$.
This implies $G_K\subset \ker \varepsilon$.
Since  $\rho$ has an abelian image,
we have $\rho(\sigma)\rho(\tau)=\rho(\tau)\rho(\sigma)$
for any $\sigma,\tau\in  G_k$, which gives 
$
(\varepsilon^{-1}(\tau)\chi_p(\tau)-\varepsilon(\tau))u(\sigma)
=(\varepsilon^{-1}(\sigma)\chi_p(\sigma)-\varepsilon(\sigma))u(\tau).
$
In particular, for  $\sigma,\tau\in  G_K$, we have 
\begin{equation}
\label{CMeq}
(\chi_p(\tau)-1)u(\sigma)
=(\chi_p(\sigma)-1)u(\tau).
\end{equation}
Take an element $\tau_0\in G_K\smallsetminus \ker \chi_p$ 
(such an element exists since $p^{\infty}$-semi-AV-tor-finiteness of $K$ 
in particular implies 
$K$ is stably $\mu_{p^{\infty}}$-finite).
By \eqref{CMeq}, we have 
$u(\sigma)=\frac{u(\tau_0)(\chi_p(\sigma)-1)}{\chi_p(\tau_0)-1}$
for any $\sigma\in G_K$.
This shows that $u=c(\chi_p-1)$ on $G_K$ for some $c\in \mbb{Q}_p$.
We see that the vector $\mbf{x}={}^{\!t}(-c,1)$ satisfies $\rho(\sigma)\mbf{x}=\mbf{x}$
for any $\sigma\in G_K$.
Hence $V_p(E)^{G_K}$ is not zero, that is, 
$E(K)[p^{\infty}]$ is infinite. This contradicts the assumption that 
$K$ is  $p^{\infty}$-semi-AV-tor-finite.
\end{proof}

\begin{proposition}
\label{residueKF}
Let $K$ be an algebraic extension of $\mbb{Q}_p$.
\begin{itemize}
\item[{\rm (1)}] 
If $K$ is Kummer-faithful, then $\mbb{F}_K$ is quasi-finite, that is, 
$G_{\mbb{F}_K}\simeq \hat{\mbb{Z}}$. 

\item[{\rm (2)}] If $K$ is AV-tor-finite,  then $\mbb{F}_K$ is finite. 
\end{itemize}
\end{proposition}

\begin{proof}
(1) Since Kummer-faithful fields over $p$-adic fields 
are stably $\mu_{\ell^{\infty}}$-finite
for any prime $\ell$ and $p^{\infty}$-semi-AV-tor-finite, 
the result follows immediately from Lemma \ref{residue}.

(2) Assume that $\mbb{F}_K$ is infinite.
There exists an increasing extensions
$\mbb{F}_{p^{m_1}}\subset \mbb{F}_{p^{m_2}}\subset \cdots $
of subfields of $\mbb{F}_K$ with $m_1<m_2<\cdots $.
This in particular implies that $\mbb{F}_K$ contains all the $(p^{m_i}-1)$-th 
roots of unity  for all $i$, and the same holds for $K$.
Hence, for the Tate curve $E_{/\mbb{Q}_p}$ associated with any choice of uniformizing element, 
$E(K)$ contains infinitely many torsion points but this  contradicts the assumption that 
$K$ is AV-tor-finite.
\end{proof}

\begin{corollary}
\label{unram}
Let $K$ be an unramified extension of some $p$-adic field. 
Then, the following are equivalnet.
\begin{itemize}
\item[{\rm (a)}] $K$ is Kummer-faithful.
\item[{\rm (b)}] $K/\mbb{Q}_p$ is quasi-finite. 
\item[{\rm (c)}] $\mbb{F}_K$ is Kummer-faithful.
\end{itemize}
\end{corollary}
\begin{proof}
It is shown by Murotani \cite[Theorem B]{Mur23b} that $\mbb{F}_K$ is Kummer-faithful
if and only if $G_{\mbb{F}_K}\simeq \hat{\mbb{Z}}$. 
By assumption on $K$, this is equivalent to say that $K/\mbb{Q}_p$ is quasi-finite. 
Thus the result immediately follows from Propositions 
and 
\ref{div}, 
\ref{fine:ppl:fine}
and 
\ref{residueKF}.
\end{proof}
\begin{remark}
Let $L$ be the completion of an algebraic extension, with finite ramification, of some $p$-adic field.
Then, it follows from \cite[Proposition 3.7]{Mur23b} that,
if the residue field of $L$ is Kummer-faithful,
then $L$ is Kummer-faithful (see also \cite[Proposition 1.4]{Tsu23}).
Thus the equivalent conditions (a), (b) and (c) 
in Corollary \ref{unram} are also equivalent to the following condition:
\begin{itemize}
\item[(d)] The completion of $K$ is Kummer-faithful. 
\end{itemize}
\end{remark}

Next we consider Kummer-faithfulness for tamely ramified extensions.

\begin{lemma}
\label{subquasi-finite}
Let $K$ be an algebraic extension of a $p$-adic field $k$.
Let $M$ {\rm (}resp.\ $N${\rm )} be the maximal unramified 
(resp.\ maximal tamely ramified) extension of $k$ contained in $K$. 
\begin{itemize}
\item[{\rm (1)}] If $K$ is Kummer-faithful,
then $M/k$ is quasi-finite. 

\item[{\rm (2)}] If $K$ is torally Kummer-faithful
and $N$ is a Galois extension of $M$,
then  $N/M$ is quasi-finite. 

\item[{\rm (3)}] If $K$  is Kummer-faithful and is a Galois extension of $k$, then $N/k$ is quasi-finite.
\end{itemize}
\end{lemma}

\begin{proof}
The assertion (1) follows from Corollary \ref{unram},
and (3) follows from (1) and (2). Thus it suffices to show (2).
Assume that there exists a prime  $\ell$ such that 
$v_{\ell}([N:M])=\infty$.
Note that  $\ell$ is not equal to $p$
since $N/M$ is totally tamely ramified. 
There exists an infinite set of finite Galois subextensions $\{M_i\}_{i\in \mbb{M}}$ in  $N/M$
with the property that 
$v_{\ell}(e_i)<v_{\ell}(e_{i+1})$ where $e_i:=[M_i:M]$.
Since $M_i$ is a Galois extension of $M$, it follows from
\cite[Chapter II, \S 5, Proposition 12]{Lan94} that
$M_i$ contains $e_i$-th roots of unity. Since $\lim_{i\to \infty}v_{\ell}(e_i)=\infty$,
it follows that $N$ contains all $\ell$-power roots of unity.
This contradicts the assumption that $K$ is torally Kummer-faithful. 
\end{proof}
\begin{corollary}
\label{tame}
Let $K$ be a Galois extension of some $p$-adic fields.
Then, the following are equivalent.
\begin{itemize}
\item[{\rm (a)}] $K/\mbb{Q}_p$ is tame and $K$ is Kummer-faithful.
\item[{\rm (b)}] $K/\mbb{Q}_p$ is quasi-finite. 
\item[{\rm (c)}] 
$K/\mbb{Q}_p$ is tame, $K$ is torally Kummer-faithful 
and $\mbb{F}_K$ is Kummer-faithful.
\end{itemize}
\begin{proof}
$(a)\Rightarrow (b)$ follows from Lemma \ref{subquasi-finite} (3).
$(b)\Rightarrow (a)$ follows from 
Propositions 
\ref{div} and 
\ref{fine:ppl:fine}.
$(a), (b)\Rightarrow (c)$ follows from Corollary \ref{unram}.
$(c)\Rightarrow (b)$ follows from Corollary \ref{unram}
and Lemma \ref{subquasi-finite} (2).
\end{proof}
\end{corollary}

\begin{proposition}
\label{KF:structure1}
Let $K$ be a Galois extension of a $p$-adic field.
Then  the following are equivalent.
\begin{itemize}
\item[{\rm (a)}] $K$ is Kummer-faithful.
\item[{\rm (b)}] $K$ has a decomposition $K=MN$. Here, 
\begin{itemize}
\item $M$ is Kummer-faithful with {\rm (}possibly infinite{\rm )} $p$-power degree over a $p$-adic field, and 
\item $N$ is a  quasi-finite Galois extension of a $p$-adic field. 
\end{itemize}
\end{itemize}
If $K$ is Kummer-faithful and is abelian extension of  a $p$-adic field,
then we can choose $N$ above
so that it is an unramified  quasi-finite extension of a $p$-adic field. 
\end{proposition}

\begin{proof}
$(b)\Rightarrow (a)$ follows from Proposition \ref{fine:ppl:fine}. 
We show $(a)\Rightarrow (b)$.
Assume that $K$ is Kummer-faithful.
Let $k$ be a $p$-adic subfield of $K$ so that $K$ is a Galois extension of $k$.
Let $N$ be the maximal tamely ramified extension of $k$ contained in $K$,
which is a Galois extension of $k$ since so is $K$. 
By Lemma \ref{subquasi-finite} (3), replacing $k$ by a finite subextension in $K$,
we may suppose that $N/k$ is prime-to-$p$.
Then, it follows from Lemma 5 of \cite{Iwa55} that the exact sequence
$1\to \mrm{Gal}(K/N)\to \mrm{Gal}(K/k)\to \mrm{Gal}(N/k)\to 1$
splits, that is, 
there exists an algebraic extension $M$ of $k$ contained in $K$
such that $K=MN$ and $M\cap N=k$. 
Since $K/N$ is a pro-$p$ extension, 
we see that the supernatural number $[M:k]$ is a power of $p$. 
This finishes a proof of $(a)\Rightarrow (b)$.
Finally we note that, under the assumption that $K$ is abelian over a $p$-adic field $k$,
the tame ramification index of $K/k$ is finite.
\end{proof}

In view of the above proposition, 
the essential difficulty in constructing a Kummer-faithful field over a $p$-adic field 
lies in the problem of whether an (almost) 
pro-$p$ algebraic extension which is Kummer-faithful can be constructed.
Later we will study some criterion of Kummer-faithfulness for Lubin-Tate extensions of $p$-adic fields.

\subsection{Finiteness of torsion points}
The purpose of this section is to give some 
equivalent conditions of Kummer-faithfulness
in terms of finiteness of torsion points.

\begin{theorem}
\label{KF:char1}
Let $\mbb{Q}_p\subset k\subset K$ be algebraic extensions.
\begin{itemize}
\item[{\rm (1)}] Assume that $k$ is Kummer-faithful and $K$ is a Galois extension of $k$.
Then, the following are equivalent.
\begin{itemize}
\item[{\rm (a)}] $K$ is Kummer-faithful.
\item[{\rm (b)}] $K$ is stably $\mu_{\ell^{\infty}}$-finite for every prime $\ell$, and 
the group $A(L)[p^{\infty}]$ is finite
for any finite extension $L/K$ 
and any abelian variety $A/L$ with good reduction.
\end{itemize}
\item[{\rm (2)}] 
Assume that $k$ is both AV-tor-finite and Kummer-faithful,  
and also assume that $K$ is a Galois extension of $k$.
Then the following are equivalent.
\begin{itemize}
\item[{\rm (a)}] $K$ is Kummer-faithful with finite residue field.
\item[{\rm (b)}] $K$ is semi-AV-tor-finite.
\item[{\rm (c)}] $K$ is AV-tor-finite.
\end{itemize}
\end{itemize}
\end{theorem}

We note that the implications $(a)\Rightarrow (b)$ in Theorem \ref{KF:char1} (1)
and $(b)\Rightarrow (c)$ in Theorem \ref{KF:char1} (2)
clearly hold for {\it every} algebraic extension field $K$ of $\mbb{Q}_p$.
If we remove the assumption ``Galois'',
then $(a)\Rightarrow (b),\ (c)$ in Theorem \ref{KF:char1} (2)
does not hold (although the author does not know about
the converse implication), see Remark \ref{galois} (2).

First we show Theorem \ref{KF:char1} (1); 
this is an immediate consequence of Proposition \ref{fine=KF}, 
Proposition \ref{residue} and Lemma \ref{ellfine} below.

\begin{lemma}
\label{ellfine}
Let $\mbb{Q}_p\subset k\subset K$ be algebraic extensions.
Let $\ell$ be a prime. 
Assume that $k$ is $\ell^{\infty}$-semi-AV-tor-finite and $K$ is a Galois extension of $k$.
Consider the following conditions.
\begin{itemize}
\item[{\rm (a)}] $K$ is $\ell^{\infty}$-semi-AV-tor-finite.
\item[{\rm (b)}] $K$ is stably $\mu_{\ell^{\infty}}$-finite, and 
the group $A(L)[\ell^{\infty}]$ is finite
for any  finite extension $L/K$ 
and any abelian variety $A_{/L}$ with good reduction.
\item[{\rm (c)}] $K$ is stably $\mu_{\ell^{\infty}}$-finite.
\end{itemize}
Then, the following hold.
\begin{itemize}
\item[{\rm (1)}] 
We have $(a) \Leftrightarrow (b)\Rightarrow (c)$.
\item[{\rm (2)}] 
If $\ell$ is not equal to $p$, then we have $(a) \Leftrightarrow (b)\Leftrightarrow (c)$.
\end{itemize}
\end{lemma}

\begin{proof}
The equivalence $(a)\Leftrightarrow (b)$ follows from essentially the same proof 
as that of Corollary 2.4 (1) of \cite{Oze23}. 
(We should give two remarks about this.
At first, in {\it loc cit.}, $k$ is assumed to be $p$-adic fields. 
However, the same argument proceeds by using $\ell^{\infty}$-semi-AV-tor-finiteness  
instead of Matuck's theorem.
Secondly,  we remark that the key for the proof is 
a non-archimedian rigid uniformization theorem.
We will use a similar method in the proof of Theorem \ref{KF:char1} (2) below.)
It suffices to show $(c)\Rightarrow (b)$ if $\ell\not=p$.
Let  $\ell\not=p$ and assume that  $K$ is stably $\mu_{\ell^{\infty}}$-finite. 
Let $L$ be a finite extension of $K$ and $A$  an abelian variety over $L$
with good reduction. Let us denote  by $\bar{A}$ the reduction of $A$. 
Since $\ell$ is prime to $p$, 
the reduction map induces a bijection between $A(L)[\ell^{\infty}]$
and $\bar{A}(\mbb{F}_{L})[\ell^{\infty}]$.
On the other hand, it follows from (c) that $\mbb{F}_{L}$ 
is a potential prime-to-$\ell$ extension over a finite field.
Thus $\bar{A}(\mbb{F}_{L})[\ell^{\infty}]$ is defined over  a finite field.
In particular, the group $\bar{A}(\mbb{F}_{L})[\ell^{\infty}]$, 
and thus also $A(L)[\ell^{\infty}]$, must be finite.
\end{proof}

Next we give a proof of Theorem \ref{KF:char1} (2).

\begin{proof}[Proof of Theorem \ref{KF:char1} (2)]
The implication $(b)\Rightarrow (c)$ is clear,
and $(c)\Rightarrow (a)$ is a result of Propositions \ref{div} and \ref{residueKF}.
We show  $(a)\Rightarrow (b)$. Assume the condition (a). 
Let  $L/K$ be a finite extension and $X/L$ a semi-abelian variety.
The goal is to show that the torsion subgroup of $X(L)$ is finite.
We will prove this in four steps depending on the situation of $X$.

{\bf Step 1.}  Consider the case where $X=T$ is a torus. 
Replacing $L$ by a finite extension,
we may assume that the torus $T$ splits over $L$.
Since $L$ is Kummer-faithful, $T(L)[\ell^{\infty}]$ is finite 
for every prime $\ell$. 
Moreover, one can check  that $T(L)[\ell]$  is trivial for 
almost all $\ell$: If not, we have $\mu_{\ell}\subset L$ 
for infinitely many primes $\ell$ but this contradicts the fact that $L$
has a finite residue field. Hence $T(L)_{\mrm{tor}}$ is finite.

{\bf Step 2.}  Consider the case where $X=A$ is an abelian variety with 
potential good reduction. Replacing $L$ by a finite extension,
we may suppose that $A$ has good reduction over $L$.
Since $L$ is AVKF, $A(L)[\ell^{\infty}]$ is finite for every prime $\ell$. 
Furthermore, since the reduction map 
from the prime-to-$p$ part $A(L)_{p'}$ of
$A(L)_{\mrm{tor}}$ is injective and the residue field of $L$ is finite, 
we see that $A(L)_{p'}$ is finite.
Hence $A(L)_{\mrm{tor}}$ is finite.

{\bf Step 3.}  Consider the case where $X=A$ is an abelian variety (without reduction hypothesis).
Let $g$ be the dimension of $A$.
Take a $p$-adic subfield $K_0$ in $L$ so that $A$ is defined over $K_0$. 
Applying a non-archimedian  rigid uniformization theorem to $A_{/K_0}$
(\cite{Ray71}, \cite{Bos91}  and \cite{BoXa96}), 
we find that there exist the following data, which is called a rigid uniformization of $A$ 
(cf.\  \cite[Definition 1.1 and Theorem 1.2]{BoXa96}):
\begin{itemize} 
\item[(i)] $S$ is a semi-abelian variety of dimension $g$ fits into an exact sequence
of $K_0$-group schemes
$0\to T \to S \to B \to 0$ where $T$ is a torus of rank $m$ and $B$
is an abelian variety which has potential good reduction,
\item[(ii)] a closed immersion of rigid $K_0$-groups $N^{\rm an}\hookrightarrow S^{\rm an}$ 
where  $N$ is a group scheme which is isomorphic to $\mbb{Z}^{\oplus m}$
 after a finite base extension. Here, the subscript  ``an'' is the GAGA functor, and  
\item[(iii)] a faithfully flat morphism $S^{\rm an}\to A^{\rm an}$ of rigid $K_0$-groups 
with kernel $N^{\rm an}$.
\end{itemize}
We have exact sequences
$$
0\to N(\overline{K}) \to S(\overline{K})\to A(\overline{K})\to 0 \quad \mrm{and} \quad 
0\to T(\overline{K}) \to S(\overline{K})\to B(\overline{K})\to 0
$$
of $G_{K_0}$-modules. 
For the proof, replacing $L$ and $K_0$ by finite extensions,
we may assume that $L$ is a Galois extension of $k$ and 
$N$ is constant over the base field $K_0$. 
We set $k_0:=kK_0$. Note that $L$ is a Galois extension of $k_0$. 
The exact sequence $0\to T(\overline{K}) \to S(\overline{K})\to B(\overline{K})\to 0$
of $G_{k_0}$-modules induces an exact sequence 
$0\to T(L)_{\mrm{tor}} \to S(L)_{\mrm{tor}}\to B(L)_{\mrm{tor}}$
of $G_{k_0}$-modules.
By  Step 1 and Step 2, it follows that
$T(L)_{\mrm{tor}}$ and $B(L)_{\mrm{tor}}$ are finite,
which shows that  $S(L)_{\mrm{tor}}$ is also finite.
Since $k$ is AV-tor-finite, there exists an integer $M>0$ 
such that both $S(L)_{\mrm{tor}}\subset S(\overline{K})[M]$
and $A(k_0)_{\mrm{tor}}\subset A(\overline{K})[M]$ hold. 
Let $n>0$ be any integer divided by $M^2$.
Since $L$ is a Galois extension of $k_0$,  
there exists an exact sequence 
$0\to S(L)[n]\to A(L)[n]\to N(\overline{K})/nN(\overline{K})$
of $G_{k_0}$-modules.
Let $P$ be an element of $A(L)[n]$. 
Since $G_{k_0}$ acts on $N(\overline{K})$ trivial,
we know that $\sigma P-P\in S(L)[n]\subset S(\overline{K})[M]$ 
for any $\sigma \in G_{k_0}$.
Thus we find $MP\in A(k_0)_{\mrm{tor}}\subset A(\overline{K})[M]$,
which implies $A(L)[n]\subset A(\overline{K})[M^2]$.
Since this relation holds for any $n$ divided by $M^2$,
we obtain that $A(L)_{\mrm{tor}}$ is contained in $A(\overline{K})[M^2]$,
which must be finite.

{\bf Step 4.}  Finally, we consider the case where $X$ is a semi-abelian variety.
We have an exact sequence
\begin{equation}
\label{semiav}
0\to T\to X \to A\to 0
\end{equation}
of group schemes over $K$, where 
 $T$ is a torus and $A$ is an abelian variety.
This induces an exact sequence
$
0\to T(L)_{\mrm{tor}}\to X(L)_{\mrm{tor}}\to A(L)_{\mrm{tor}}
$
of abelian groups.
Since $T(L)_{\mrm{tor}}$ and $A(L)_{\mrm{tor}}$ are 
finite by Step 1 and Step 3, 
we conclude that $X(L)_{\mrm{tor}}$ is finite as desired.
\end{proof}

\begin{remark} 
\label{galois}
(1) 
Let  $\mbb{Q}_p\subset k\subset K$ be as in Theorem \ref{KF:char1} (2).
The equivalent conditions (a), (b) and (c) in the theorem 
are also equivalent to the following condition (d):
\begin{itemize}
\item[{\rm (d)}]  
For every finite extension $L$ of $K$ and every commutative algebraic group 
$G$ over $L$, it holds that 
$G(L)_{\mrm{tor}}$ is finite. 
\end{itemize}
In fact, we can check the equivalence between (d) and the others 
by almost the same method as the proof of the theorem; 
we should use Chevalley's decomposition of 
(possibly non-connected) commutative algebraic groups
(cf.\ \cite[Theorem 2.9]{Bri17})  instead of \eqref{semiav} 
in the argument of Step 4.

(2) We can not remove the assumption "Galois" from the statement of Theorem \ref{KF:char1} (2). 
Let $\mfrak{n}=\prod_{\ell:\mrm{prime}}\ell$ be a supernatural number and 
take a compatible system $(p_k)_{k\mid \mfrak{n}}$ of $(p^{k}-1)$-th roots $p_k$ of $p$.
Let $K$ be the extension field $\mbb{Q}_p(p_k; k\mid \mfrak{n})$
of $\mbb{Q}_p$ obtained by adjoining all $p_k$ for all $k\mid \mfrak{n}$.  
We see that $K$ is totally ramified over $\mbb{Q}_p$ and 
the Galois closure $\hat{K}$ of $K$ over $\mbb{Q}_p$ is quasi-finite over $\mbb{Q}_p$.
The field $\hat{K}$ is  locally semi-AV-tor-finite by Proposition \ref{fine:ppl:fine}.
In particular, $\hat{K}$ is Kummer faithful by Proposition \ref{fine=KF} (2)
and thus so is $K$.
On the other hand, $K$ is not AV-tor-finite 
since there are infinitely many $K$-rational torsion points of  
the Tate curve with uniformizing parameter $p$. 
\end{remark}

\section{Kummer-faithfulness of Lubin-Tate extensions}
\label{KFLT}

The aim of this section is to give some criteria on 
Kummer-faithfulness of Lubin-Tate extensions of $p$-adic fields.
Throughout this section, we fix an algebraic closure $\overline{\mbb{Q}}_p$ of $\mbb{Q}_p$.
Unless otherwise specified, all algebraic extension fields of $\mbb{Q}_p$ appearing in this section are subfields of 
our fixed $\overline{\mbb{Q}}_p$. 
Let $k$ be a $p$-adic field and $\pi$ a uniformizer.
Let $F_{\pi}$ be the Lubin-Tate formal group 
over the integer ring of $k$ associated with $\pi$.
Let $k_{\pi}$ be the extension of $k$
obtained by adjoining all $\pi$-power torsion points of $F_{\pi}$,
which is called the Lubin-Tate extension of $k$ associated with $\pi$.
Local class field theory asserts that $k_{\pi}$ is  a maximal totally ramified 
abelian extension of $k$, 
and the composite of $k_{\pi}$ and the maximal unramified 
extension field $k^{\mrm{ur}}$ of $k$ 
coincides with the maximal abelian extension field 
$k^{\mrm{ab}}$ of $k$.
We should remark that the following conditions on 
the Lubin-Tate extension field $k_{\pi}$ are equivalent by Theorem \ref{KF:char1} (2):
\begin{itemize}
\item[(a))] $k_{\pi}$ is Kummer-faithful. 
\item[(b)] $k_{\pi}$ is semi-AV-tor-finite.
\item[(c)] $k_{\pi}$ is AV-tor-finite.
\end{itemize}

We are interested in determining which Lubin–Tate extension fields $k_{\pi}$ 
satisfy the three equivalent conditions mentioned above.
Before the main part of this section,
we note that it is easy to check whether 
$k_{\pi}$ is
torally Kummer-faithfulness or not.
\begin{proposition}
\label{TKF}
$k_{\pi}$ is torally Kummer-faithful if and only if $q^{-1}\mrm{Nr}_{k/\mbb{Q}_p}(\pi)\not\in \mu_{p-1}$.    
\end{proposition}
\begin{proof}
 One verifies immediately that $k_{\pi}$ is torally Kummer-faithful if and only if 
it is stably $\mu_{p^{\infty}}$-finite by Proposition \ref{div}. 
The latter condition is equivalent to $q^{-1}\mrm{Nr}_{k/\mbb{Q}_p}(\pi)\not\in \mu_{p-1}$
by local class field theory.
\end{proof}

\subsection{Criterion for Lubin-Tate extensions}
In this section, we give a proof of Theorem \ref{LT} in the Introduction.
In fact, we prove more precise criterion for Kummer-faithfulness of Lubin-Tate extension fields.
To simplify statements,
we use the following notation.
Let $k$ be a $p$-adic field and 
$\pi$ a uniformizer of $k$. 
For a tuple $(r_{\sigma})_{\sigma \in \Gamma_k}$
with $r_{\sigma}\in \{0,1\}$,
we often set 
$$
\hat{\pi}
:=\prod_{\sigma \in \Gamma_k} \sigma \pi^{r_{\sigma}}
$$ 
(of course this notation depends on the choices of 
$\pi$ and $(r_{\sigma})_{\sigma \in \Gamma_k}$).
We recall that $W(q)$ is the set of Weil $q$-integers.

\begin{theorem}
\label{refined:LT}
Let $k$ be  a $p$-adic field with residue cardinality $q=p^f$ 
and $\pi$ a uniformizer of $k$.
We denote by $k_{\pi}$ the Lubin-Tate extension field of $k$ associated with $\pi$.
\begin{itemize}
\item[{\rm (1)}]  
If $k$ is a Galois extension of $\mbb{Q}_p$ and $k_{\pi}$ is not Kummer-faithful, 
then either of the following holds.
\begin{itemize}
\item[{\rm (a)}] $q^{-1}\mrm{Nr}_{k/\mbb{Q}_p}(\pi)\in \mu_{p-1}$.
\item[{\rm (b)}] 
For some $(r_{\sigma})_{\sigma \in \Gamma_k}$ with $r_{\sigma}\in \{0,1\}$, 
it holds $\hat{\pi}\in W(q)$.
\end{itemize}

\item[{\rm (2)}] 
If $q^{-1}\mrm{Nr}_{k/\mbb{Q}_p}(\pi)\in \mu_{p-1}$ or   
$\mrm{Nr}_{k/\mbb{Q}_p}(\pi)\in   W(q)$, 
then $k_{\pi}$ is not Kummer-faithful.

\item[{\rm (3)}] Assume that $k$ is a Galois extension of $\mbb{Q}_p$, and also
assume that all the following conditions hold.
\begin{itemize}
\item[{\rm (i)}] For some $(r_{\sigma})_{\sigma\in \Gamma_k}$
with $r_{\sigma}\in \{0,1\}$, it holds 
$\hat{\pi}\in W(q)$.
\item[{\rm (ii)}] $\pi^r\in \mbb{Q}_p$ for some integer $r\ge 1$.
\item[{\rm (iii)}]
Let $\hat{\pi}$ be as in {\rm (i)}.
\begin{enumerate}
\item $f$  divides $\mrm{ord}_v(\hat{\pi})f_v$
for any finite place $v$ of $\mbb{Q}(\hat{\pi})$ above $p$.
Here, $\mrm{ord}_v$ is the valuation associated with $v$ normalized by $\mrm{ord}_v(\mbb{Q}(\hat{\pi})^{\times})=\mbb{Z}$
and $f_v$ is the residue degree of $\mbb{Q}(\hat{\pi})/\mbb{Q}$ at $v$.
\item  Let $f(X)\in \mbb{Q}[X]$ be the minimal  polynomial of $\hat{\pi}$ over $\mbb{Q}$. Then, 
the degree of any irreducible factor  of $f(X)$  in $\mbb{Q}_p[X]$ 
is not of the form $n\cdot [\mbb{Q}_p(\hat{\pi}):\mbb{Q}_p]$ 
for any integer $n\ge 2$.
\end{enumerate}
\end{itemize}
Then, $k_{\pi}$ is not Kummer-faithful. 
\end{itemize}
\end{theorem}

The aim of this section is to prove the above theorem.
Our argument relies heavily on both Honda–Tate theory and $p$-adic Hodge theory.
We begin by briefly reviewing the aspects of these theories that are essential for our proof.
Let $A$ be a simple abelian variety over $\mbb{F}_q$ of dimension $g>0$
(here, ``simple" stands for not ``$\overline{\mbb{F}}_q$-simple" but ``$\mbb{F}_q$-simple").
Put $D=\mrm{End}_{\mbb{F}_q}(A) \otimes_{\mbb{Z}} \mbb{Q}$ and 
denote by $\pi_A \in D$ the geometric Frobemius endomorphism of $A_{/\mbb{F}_q}$.
Weil showed that $\pi_A$ is a Weil $q$-integer. 
The degree $2g$ characteristic polynomial $f_{\bar{A}/\mbb{F}_q}\in \mbb{Z}[T]$ 
of $\pi_A$ has a single monic irreducible factor over $\mbb{Q}$,
and  $f_{\bar{A}/\mbb{F}_q}$  coincides with 
the characteristic polynomial 
the action of the arithmetic $q$-Frobenius element of $G_{\mbb{F}_q}$ on 
the $\ell$-adic Tate module $V_{\ell}(A)$ of $A$ for any prime $\ell\not=p$.

\begin{theorem}
\label{Honda-Tate}
Let $\mbb{F}$ be the finite field of order $q$.
\begin{itemize}
\item[{\rm (1)}] The assignment $A\mapsto \pi_A$  defines a bijection from the set of 
isogeny classes of simple abelian varieties over $\mbb{F}$ to 
the set of $G_{\mbb{Q}}$-conjugacy classes of Weil $q$-integers. 
\item[{\rm (2)}] 
Let $A$ be an  abelian vareity over $\mbb{F}$
which is $\mbb{F}$-isogenous to a power of an $\mbb{F}$-simple abelian variety.
Then there exist a finite extension $\mbb{F}'/\mbb{F}$
and a $p$-adic field $k'$ with residue field $\mbb{F}'$
such that  $A$ is $\mbb{F}'$-isogenous to the reduction of 
a CM abelian variety with good reduction over $k'$. 
\end{itemize}
\end{theorem}
The first assertion is a central result of Honda–Tate theory,
while the second is commonly referred to as the "Honda–Tate CM lifting theorem".
For more details, we refer the reader to Section 1.6 of \cite{CCO14}.

Next, we summarize some fundamental facts about $p$-adic Hodge theory.
For an introduction to the basic notions of $p$-adic Hodge theory,
it is helpful for the reader to refer  \cite{Fon94a} and \cite{Fon94b}.  
Let $K$ be a $p$-adic field.
Below we often use the following notations.
Let $B_{\mrm{cris}}$ be the Fontaine's $p$-adic period ring
and set $D^K_{\mrm{cris}}(V) := (B_{\mrm{cris}}\otimes_{\mbb{Q}_p} V)^{G_K}$
for any $\mbb{Q}_p$-representation $V$ of $G_K$.
Let us denote by $K_0$ the maximal unramified subextension of $K/\mbb{Q}_p$.
Since $B_{\mrm{cris}}^{G_K}=K_0$,  $D^K_{\mrm{cris}}(V)$ is a $K_0$-vector space.
Denote by $\vphi_{K_0}$ the arithmetic Frobenius map of $K_0$, that is, 
the (unique) lift of the $p$-th power map on the residue field of $K_0$.
This extends to $B_{\mrm{cris}}$, and moreover this extended Frobenius map induces 
a $K_0$-semi-liniear endomorphism
$\vphi$ on $D^K_{\mrm{cris}}(V)$, 
which is also called Frobenius endomorphism.
Note that the $f_K$-th iterate $\vphi^{f_K}$ of $\vphi$ on $D^K_{\mrm{cris}}(V)$
is $K_0$-linear
where $f_{K}$ is the residue degree of $K/\mbb{Q}_p$.
It is known that 
$D^K_{\mrm{cris}}(V)$ is of finite dimensional over $K_0$ with 
$\dim_{K_0} D^K_{\mrm{cris}}(V)\le \dim_{\mbb{Q}_p} V$.
We say that $V$ is crystalline if the equality 
$\dim_{K_0} D^K_{\mrm{cris}}(V)=\dim_{\mbb{Q}_p} V$ holds.
Let us recall locally algebraic theory for crystalline characters; 
see \cite[Appendix of Chapter III]{Ser98} and   \cite[Appendix B]{Con11}
for more precise information.
Let $F$ be a $p$-adic field and  $\psi\colon G_K\to F^{\times}$ a continuous character.
We denote by $F(\psi)$ the $\mbb{Q}_p$-representation of $G_K$
underlying a $1$-dimensional $F$-vector space endowed with an $F$-linear action by $G_K$ via $\psi$.
We say that $\psi$ is  crystalline   if $F(\psi)$ is crystalline. 
Suppose that $\psi$ is crystalline.
For any $\sigma\in \Gamma_F$, let $\chi_{\sigma F}\colon I_{\sigma F}\to \sigma F^{\times}$ 
be the restriction to the inertia $ I_{\sigma F}$ of the Lubin-Tate character associated  
with any choice of uniformizer of $\sigma F$ 
(it depends on the choice of a uniformizer of $\sigma F$,
but its restriction to the inertia subgroup does not).
Assume that $K$ contains the Galois closure of $F/\mbb{Q}_p$. 
Then, we have 
\begin{equation}
\label{psiHT}
\psi = 
\prod_{\sigma \in \Gamma_F} \sigma^{-1} \circ \chi_{\sigma F}^{h_{\sigma}}
\end{equation}
on the inertia $I_K$ for some integer $h_{\sigma}$ 
where $\Gamma_F:=\mrm{Hom}_{\mbb{Q}_p}(F,\overline{\mbb{Q}}_p)$.
Note that $\{h_{\sigma} \mid \sigma\in \Gamma_F \}$ 
is the (multi-)set of Hodge-Tate weights of $F(\psi)$, 
that is, $C\otimes_{\mbb{Q}_p} F(\psi)\simeq \oplus_{\sigma \in \Gamma_F} C(h_{\sigma})$
where $C$ is the completion of $\overline{\mbb{Q}}_p$.

Here we recall a remarkable observation of Conrad for crystalline characters.
We denote by $\underline{K}^{\times}$ 
the Weil restriction  $\mrm{Res}_{K/\mbb{Q}_p}(\mbb{G}_m)$.
This is an algebraic torus such that, for a $\mbb{Q}_p$-algebra $R$,
the $R$-valued points $\underline{K}^{\times}(R)$ of $\underline{K}^{\times}$ 
is $\mbb{G}_m(R\otimes_{\mbb{Q}_p} K)$. 
If $\psi\colon G_K\to F^{\times}$ is 
a continuous character, we may regard $\psi$ as a character of 
$\mrm{Gal}(K^{\mrm{ab}}/K)$
where $K^{\mrm{ab}}$ is the maximal abelian extension of $K$.
We denote by $\psi_K\colon K^{\times}\to F^{\times}$
the composite of 
the reciprocity map $K^{\times}\to \mrm{Gal}(K^{\mrm{ab}}/K)$ of local class field theory 
and 
$\psi\colon \mrm{Gal}(K^{\mrm{ab}}/K)\to F^{\times}$. 
For example, if $\chi\colon G_K\to \mbb{Q}_p^{\times}$
is the $p$-adic cyclotomic character,
we have $\chi_K(x)=\left(\mrm{Nr}_{K/\mbb{Q}_p}(x)|\mrm{Nr}_{K/\mbb{Q}_p}|_p\right)^{-1}$
for $x\in K^{\times}$
where $|\cdot|_p$ is the $p$-adic absolute value normalized by $|p|_p=p^{-1}$.
The following  result is shown in 
Proposition B.4 of \cite{Con11}.
\begin{proposition}
\label{Conrad}
Let $\psi\colon G_K\to F^{\times}$ be 
a continuous character.
\begin{itemize}
\item[{\rm (1)}] 
$F(\psi)$ is crystalline if and  only if 
there exists a {\rm(}necessarily unique{\rm )} $\mbb{Q}_p$-homomorphism 
$\psi_{\mrm{alg},K}\colon \underline{K}^{\times}\to \underline{F}^{\times}$
such that $\psi_K$ and $\psi_{\mrm{alg},K}$ {\rm (}on $\mbb{Q}_p$-points{\rm )} 
coincides on $\cO_K^{\times} (\subset \underline{K}^{\times}(\mbb{Q}_p))$.
\item[{\rm (2)}]
Assume that $F(\psi)$ is crystalline and let $\psi_{\mrm{alg},K}$ be as in {\rm (1)}. 
Then,  
$D^K_{\mrm{cris}}(F(\psi^{-1}))$ 
is free of rank $1$ over $K_0\otimes_{\mbb{Q}_p} F$ 
and its $K_0$-linear endomorphism $\vphi^{f_K}$ 
is given by the action of the product 
$$
\psi_K(\pi_K)\cdot \psi_{\mrm{alg},K}^{-1}(\pi_K)\in F^{\times},
$$
where $\pi_K$ is any uniformizer of $K$
and $f_{K}$ is the residue degree of $K/\mbb{Q}_p$.
Note that this product is independent of the choice of $\pi_K$.
\end{itemize}
\end{proposition}

Motivated by the above proposition,
for  a crystalline character  $\psi\colon G_K\to F^{\times}$  of $G_K$, 
we set 
$$
\Phi_K(\psi):=\psi_K(\pi_K)\cdot \psi_{\mrm{alg},K}^{-1}(\pi_K)\in F^{\times}.
$$
This is invariant under the coefficient field extension $F'/F$.
For a finite extension $K'/K$, we set $\Phi_{K'}(\psi):=\Phi_{K'}(\psi|_{G_{K'}})\in F^{\times}$.
For example,  one verifies  $\Phi_K(\chi_{\pi})=\pi^{f_{K/k}}$ where $\chi_{\pi}\colon G_k\to k^{\times}$ is 
the Lubin-Tate character of $k$  associated with a uniformizer $\pi$
and 
$K$ is a finite extension of $k$ with residue degree $f_{K/k}$
(see also Lemma \ref{Phi} (2) below).  
Note that the character
$\psi_{{\rm alg},K}$ on $\mbb{Q}_p$-points  coincides with 
$\prod_{\sigma \in \Gamma_F} \sigma^{-1} \circ \mrm{Nr}_{K/\sigma F}^{-h_{\sigma}}$
if $K$ contains the Galois closure of $F/\mbb{Q}_p$ and $\psi$ is of the form \eqref{psiHT}. 
By Proposition \ref{Conrad} (2), one can verify 
that the roots of the characteristic polynomial of the $\vphi^{f_K}$-action on the $K_0$-vector space 
$D^K_{\mrm{cris}}(F(\psi^{-1}))$ are the $\mbb{Q}_p$-conjugates of $\Phi_K(\psi)$, that is,
$$
\det (T-\vphi^{f_K}\mid D^K_{\mrm{cris}}(F(\psi^{-1})))=\prod_{\sigma\in \Gamma_F} (T-\sigma \Phi_K(\psi)).
$$
We summarize some basic properties of $\Phi_K(\psi)$.
\begin{lemma}
\label{Phi}
Let  $\psi\colon G_K\to F^{\times}$ be a crystalline character
and $\{h_{\sigma}\}_{\sigma\in \Gamma_F}$ the {\rm (}multi-{\rm )}set of Hodge-Tate 
weights of $F(\psi)$.

\begin{itemize}
\item[{\rm (0)}]
The $\vphi^{f_K}$-action on the free $K_0\otimes_{\mbb{Q}_p} F$-module 
$D^K_{\mrm{cris}}(F(\psi^{-1}))$ of rank $1$ 
is given $\Phi_K(\psi)$.

\item[{\rm (1)}] For a crystalline character $\psi'\colon G_K\to F^{\times}$, we have 
$\Phi_K(\psi \psi')=\Phi_K(\psi)\Phi_K(\psi')$.

\item[{\rm (2)}] 
For a finite extension $K'/K$ with residue degree $f_{K'/K}$, 
we have $\Phi_{K'}(\psi)=\Phi_K(\psi)^{f_{K'/K}}$.

\item[{\rm (3)}] For $\sigma\in \Gamma_F$, we have 
$\Phi_K(\sigma\psi)=\sigma \Phi_K(\psi)$.

\item[{\rm (4)}] $\Phi_K(\psi)\in \cO_F^{\times}$ if and only if  $\sum_{\sigma\in \Gamma_F} h_{\sigma}=0$.

\item[{\rm (5)}] Assume either ``$h_{\sigma}\ge 0$ for all $\sigma$" or ``$h_{\sigma}\le 0$ for all $\sigma$".
If $\Phi_K(\psi)=1$, then $\psi=1$.

\item[{\rm (6)}] Let $\pi_F$ be a uniformizer of $F$ with the property that $\pi_F^r\in \mbb{Q}_p$ for some $r>0$
(such $\pi_F$ exists if $F/\mbb{Q}_p$ is tamely ramified).
If $K\supset F$, $F$ is a Galois extension of $\mbb{Q}_p$ and  $\Phi_K(\psi)=1$, then we have
$$
\psi = 
\prod_{\sigma \in \Gamma_F} \sigma^{-1} \circ \chi_{\pi_F}^{h_{\sigma}}
$$
on an open subgroup of $G_K$.
Moreover, 
\begin{itemize}
\item[{\rm (i)}] $\psi\colon G_K\to F^{\times}$ extends to $G_F$, and 
\item[{\rm (ii)}] there exists a character $\epsilon\colon G_F\to \overline{\mbb{Q}}_p^{\times}$ 
 such that $\epsilon^r=1$ and 
$
\psi = \epsilon \cdot
\prod_{\sigma \in \Gamma_F} \sigma^{-1} \circ \chi_{\pi_F}^{h_{\sigma}}
$ on $G_F$.
\end{itemize}
\end{itemize}
\end{lemma}

\begin{proof}
(0) is Proposition \ref{Conrad} (2).  
(1) and (3) are clear from the definition of $\Phi_K$.
We consider (2). By local class field theory, 
$\psi_{K'}\colon (K')^{\times}\to F^{\times} $ is the composite of the norm map 
$\mrm{Nr}_{K'/K}\colon K'^{\times}\to K^{\times}$  and 
$\psi_K\colon K^{\times}\to F^{\times}$, and also
$\psi_{\mrm{alg},K'}\colon \underline{K'}^{\times} \to \underline{F'}^{\times}$
 is the composite of the norm map 
$\mrm{Nr}_{K'/K}\colon \underline{K'}^{\times} \to \underline{K}^{\times} $ and 
$\psi_{{\rm alg},K}\colon \underline{K}^{\times} \to \underline{F}^{\times} $.
Take uniformizers $\pi_K$ and $\pi_{K'}$ of $K$ and $K'$, respectively.
Writing $f=f_{K'/K}$ and $\mrm{Nr}_{K'/K}(\pi_{K'})=\pi_K^f u$ for some $u\in \cO_K^{\times}$,
we have 
$\Phi_{K'}(\psi)=\psi_{K'}(\pi_{K'})\cdot \psi_{\mrm{alg},K'}^{-1}(\pi_{K'})
=\psi_{K}(\pi_{K}^f u)\cdot \psi_{\mrm{alg},K}^{-1}(\pi_K^f u)
=(\psi_{K}(\pi_{K})\cdot \psi_{\mrm{alg},K}^{-1}(\pi_K))^f
=\Phi_{K}(\psi)^f$. This shows (2).
Next we show (4). 
By (2), 
we may assume that $K$
contains the Galois closure of $F/\mbb{Q}_p$. 
Thus $\psi$ is of the form 
\eqref{psiHT} on $I_K$.
Note that 
$\Phi_K(\psi)\in \cO_F^{\times}$ if and only if $\psi_{\mrm{alg},K}(\pi_K)\in \cO_F^{\times}$
since $\psi_K$ has values in $\cO_F^{\times}$.
Since $\psi_{\mrm{alg},K}(\pi_K)$  coincides with 
$\prod_{\sigma \in \Gamma_M} \sigma^{-1} \mrm{Nr}_{K/\sigma F}(\pi_K)^{-h_{\sigma}}$,
we see that $\psi_{\mrm{alg},K}(\pi_K)$ is a $p$-adic unit if and only if 
$\sum_{\sigma\in \Gamma_F} h_{\sigma}=0$, which shows (4). 
We show (5). It follows from the assumption and  (4) that $h_{\sigma}=0$ for all $\sigma$.
Thus $\psi_{\mrm{alg},K}$ is trivial  and then $\psi_K$ is trivial on $\cO_K^{\times}$.
Furthermore, $\Phi_K(\psi)=1$ implies $\psi_K(\pi_K)=1$. Hence we have $\psi_K=1$ on $K^{\times}$,
equivalently, $\psi=1$. 
Finally, we show (6). 
By $\Phi_K(\psi)=1$, we have 
\begin{equation}
\label{psieq}
\psi_K(x)=\prod_{\sigma \in \Gamma_F} \sigma^{-1} \circ \mrm{Nr}_{K/F}(x)^{-h_{\sigma}}
\end{equation}
if $x$ is any uniformizer $\pi_K$ of $K$.
It follows from the definition of $\{h_{\sigma}\}_{\sigma\in \Gamma_F}$ 
that \eqref{psieq} also holds for any $x\in \cO_K^{\times}$.
Hence the equality \eqref{psieq} holds for every $x\in K^{\times}$.
Thus $\psi\colon G_K\to F^{\times}$ extends to $G_F$ so that 
$$
\psi_F(x)=\prod_{\sigma \in \Gamma_F} \sigma^{-1}x^{-h_{\sigma}}
$$ 
for $x\in F^{\times}$.
Taking $\pi_F$  as in the statement of (6),  by (4), we see 
$$
\psi_F(\pi_F^r)=\prod_{\sigma \in \Gamma_F} (\sigma^{-1}\pi_F^r)^{-h_{\sigma}}
=\prod_{\sigma \in \Gamma_F} \pi_F^{-r h_{\sigma}}=1.
$$
Hence we have 
$
\psi^r= \left(\prod_{\sigma \in \Gamma_F} \sigma^{-1} \circ \chi_{\pi_F}^{h_{\sigma}}\right)^r
$ on $G_F$. Now the result immediately follows.
\end{proof}

Now we are ready to prove Theorem \ref{refined:LT}.

\begin{proof}[Proof of Theorem \ref{refined:LT}]
(1) The assertion is an immediate consequence of \cite[Theorem 1.2]{Oze23}.
However, we include a proof here for the sake of completeness
(in fact,  arguments will be simpler since we only need to consider  abelian varieties here).
Assume that $k$ is a Galois extension of $\mbb{Q}_p$ and 
$k_{\pi}$ is not Kummer-faithful.
By Theorem \ref{KF:char1}, we know that 
(i) $k_{\pi}$ is not stably $\mu_{p^{\infty}}$-finite, or,
(ii) the group $A(L)[p^{\infty}]$ is infinite
for some finite extension $L/k_{\pi}$ 
and some abelian variety $A_{/L}$ with good reduction.
In the former case (i), we have 
$q^{-1}\mrm{Nr}_{k/\mbb{Q}_p}(\pi)\in \mu_{p-1}$. 
In the rest of the proof we assume the latter case (ii) and 
let $L/k_{\pi}$ and $A_{/L}$ 
be as in (ii).
Take a finite subextension  $K/k$ in  $L/k$ so that 
$L=Kk_{\pi}$, $A$ is defined over $K$ and  $A$ has good reduction over $K$.
Since  $L$ is a Galois extension of $K$ and the group $A(L)[p^{\infty}]$ is infinite, 
 $V:=V_p(A)^{G_L}$ is a non-zero 
$G_K$-stable submodule of the $p$-adic Tate module $V_p(A)$ of $A$.
Since  the $G_K$-action on $V$ factors through an abelian quotient, 
by taking a $p$-adic field $F$ large enough,
we have an isomorphism 
$$
(V\otimes_{\mbb{Q}_p} F)^{\mrm{ss}}\simeq F(\psi_1)\oplus F(\psi_2)\oplus 
\cdots \oplus F(\psi_{t})
$$ 
of $F[G_K]$-modules for some continuous characters $\psi_i\colon G_K\to F^{\times}$.
Here, ``ss" stands for the semi-simplification as a $F[G_K]$-module.
Each $\psi_i$ is crystalline since $A$ has good reduction
(see \cite[Section 6]{Fon82}; see also \cite[Theorem 1]{CI99}).
Furthermore, each $\psi_i$ factors through $\mrm{Gal}(L/K)$.
Replacing $L,K$ and $F$ by finite extensions,
by Lemma 2.5 of \cite{Oze23}, 
we find that 
$
\psi_i=\prod_{\sigma\in \Gamma_k} \sigma \circ \chi_{\pi}^{r_{i,\sigma}}
$
on $G_K$ for some $r_{i,\sigma}\in \{0,1 \}$
(here, we need the assumption that $k$ is a Galois extension of $\mbb{Q}_p$
to apply the lemma).
Thus  
\begin{equation}
\label{Phieq}
\Phi_K(\psi_i)
=\prod_{\sigma\in \Gamma_k} \sigma \Phi_K(\chi_{\pi})^{r_{i,\sigma}} 
=\left(\prod_{\sigma\in \Gamma_k} \sigma \pi^{r_{i,\sigma}}\right)^{f_{K/k}}.
\end{equation}
Now we set 
$$
f(T):=\det (T-\vphi^{f_K}\mid D^K_{\mrm{cris}}(V^{\vee}))
$$
where $f_{K}$ is the residue degree of $K/\mbb{Q}_p$ and 
``$\vee$" stands for the dual of $\mbb{Q}_p[G_K]$-modules.
We have 
$$
f(T)^{[F:\mbb{Q}_p]}
= \mrm{det}(T-\varphi^{f_K}\mid D^K_{\mrm{cris}}(V^{\vee} \otimes_{\mbb{Q}_p} F)) 
 = \prod^{t}_{i=1}\prod_{\sigma\in \Gamma_F} (T-\sigma \Phi_K(\psi_i))
$$
by Proposition \ref{Conrad} (2). 
The  polynomial  $f(T)$ 
is a divisor of $f_A(T):=\mrm{det}(T-\varphi^{f_K}\mid D^K_{\mrm{cris}}(V_p(A)^{\vee}))$
and it follows from $p$-adic Hodge theory (cf.\ \cite{Fal88} and \cite{CLS98}) 
that 
\begin{equation}
\label{pHT}
f_A(T)
=\mrm{det}(T-\varphi^{f_K}\mid D^K_{\mrm{cris}}
(H^1_{\aet}(A_{\overline{K}},\mbb{Q}_p)) 
= \mrm{det}(T-\mrm{Frob}^{-1}_{K}\mid 
H^1_{\aet}(A_{\overline{K}},\mbb{Q}_{\ell}) )
\end{equation}
for any prime $\ell\not=p$, where 
$\mrm{Frob}_{K}$ stands for the arithmetic Frobenius of $K$. 
By the Weil conjecture, 
we conclude that each $\Phi_K(\psi_i)$ is a Weil $q_K$-integer.
It follows from \eqref{Phieq} that 
$\prod_{\sigma\in \Gamma_k} \sigma \pi^{r_{i,\sigma}}$ is 
a Weil $q$-integer.

(2) If $q^{-1}\mrm{Nr}_{k/\mbb{Q}_p}(\pi)$ is a root of unity,
then $k_{\pi}$ is not torally-Kummer faithful by Proposition \ref{TKF} and thus it is not Kummer-faithful.
We consider the case where $\mrm{Nr}_{k/\mbb{Q}_p}(\pi)\in   W(q)$. It suffices to show that 
there exists a $p$-adic field $K$ and an abelian variety defined over $K$ 
which has  infinitely many $Kk_{\pi}$-rational $p$-power torsion points.
Let $\bar{A}$ be  the   simple abelian variety defined over $\mbb{F}_q$ 
(uniquely determined up to $\mbb{F}_q$-isogeny)
which corresponds to the Weil $q$-number  
$\mrm{Nr}_{k/\mbb{Q}_p}(\pi)$
via Honda-Tate theory (cf.\ Theorem \ref{Honda-Tate} (1)).
(Note that $\bar{A}$ is ordinary by the first Theorem of \cite[Section III]{MiWa71}.)
Moreover, by Honda-Tate lifting theorem (cf.\ Theorem \ref{Honda-Tate} (2)), 
there exists a finite extension $K/k$ 
and a CM abelian variety $B$ over $K$ with good reduction  
such that $\bar{A}$ is isogenous to the reduction  $\bar{B}$ of $B$
over the residue field of $K$. 
Put $g=\dim B$ and denote by $L$ the CM field of $B$
(so there exists an embedding from $L$ into $\mrm{End}_K(B)\otimes_{\mbb{Z}} \mbb{Q}$).
Let $\prod_i L_i$ denote the decomposition of $L\otimes_{\mbb{Q}} \mbb{Q}_p$
into a finite product of 
finite extensions of $\mbb{Q}_p$
(note that, a priori, 
each factor field $L_i$ does not live in the fixed algebraically closed field $\overline{\mbb{Q}}_p$).
Since  the $G_K$-action on $V_p(B)$  commutes with $L\otimes_{\mbb{Q}} \mbb{Q}_p$-action,
the above decomposition gives the decomposition 
$V_p(B)\simeq \oplus_i V_i$ of $G_K$-modules with the property that 
each $V_i$ is naturally equipped with a structure of one dimensional 
$L_i$-representation of $G_K$.
By choosing a $\mbb{Q}_p$-algebra embedding 
from $L_i$ into our fixed $\overline{\mbb{Q}}_p$ for each $i$,
we may regard $L_i$ as a subfield of $\overline{\mbb{Q}}_p$.
Denote by $\psi_i\colon G_K\to L_i^{\times}$
the character obtained by the $L_i$-linear $G_K$-action on $V_i$. 
Each $\psi_i$ is crystalline since $B$ has good reduction
(note that the $L_i$-representation $L_i(\psi_i)$ considered as a $\mbb{Q}_p$-representation 
is isomorphic to $V_i$, which is a $\mbb{Q}_p$-subrepresentation of $V_p(B)$).
It holds that 
\begin{equation}
\label{Vi}
\det (T-\vphi^{f_K}\mid D^K_{\mrm{cris}}(V_i^{\vee}))
=\prod_{\tau \in \Gamma_{L_i}} (T-\Phi_K(\tau\circ \psi_i))
\end{equation}
by Proposition \ref{Conrad} (2) and 
Lemma \ref{Phi} (3). 
Now we set 
$$
f_B(T):=\det (T-\vphi^{f_K}\mid D^K_{\mrm{cris}}(V_p(B)^{\vee})).
$$
Note that $f_B(T)$ is equal to  the characteristic polynomial 
$\mrm{det}(T-\mrm{Frob}_{K}\mid V_{\ell}(B))
=\mrm{det}(T-\mrm{Frob}_{K}\mid V_{\ell}(\bar{B}))$ for any prime $\ell\not=p$ (see \eqref{pHT}), 
which also coincides with the characteristic polynomial 
$\mrm{det}(T-\mrm{Frob}_{K}\mid V_{\ell}(\bar{A}))$.
Hence $\mrm{Nr}_{k/\mbb{Q}_p}(\pi)^{f_{K/k}}$ is a root of $f_B(T)$
by the choice of $\bar{A}$.
Since $f_B(T)$ is the product of the polynomial \eqref{Vi} over all $i$,  
we have an equality 
$\mrm{Nr}_{k/\mbb{Q}_p}(\pi)^{f_{K/k}}=\Phi_K(\tau\circ \psi_i)$ 
for some $i$ and some $\tau\in \Gamma_{L_i}$. 
Since $\mrm{Nr}_{k/\mbb{Q}_p}(\pi)$ is an element of $\mbb{Q}_p$, 
we have $\mrm{Nr}_{k/\mbb{Q}_p}(\pi)^{f_{K/k}}
=\Phi_K(\psi_i)$ by Lemma \ref{Phi} (3).
Note that 
$
 \mrm{Nr}_{k/\mbb{Q}_p}(\pi)^{f_{K/k}}
= \left(\Phi_k({\mrm{Nr}_{k/\mbb{Q}_p} \circ \chi_{\pi}})\right)^{f_{K/k}}
=\Phi_K({\mrm{Nr}_{k/\mbb{Q}_p} \circ \chi_{\pi}}) 
$
since $\Phi_k(\chi_{\pi})=\pi$.
Hence we obtain 
\begin{equation}
\label{PsiK}
\Phi_K(\psi_i^{-1}\cdot ({\mrm{Nr}_{k/\mbb{Q}_p} \circ \chi_{\pi}}))=1.
\end{equation}
On the other hand, replacing $K$ by a finite extension 
so that $K$ contains the Galois closure of $L_i/\mbb{Q}_p$, 
we find 
$\psi_i=\prod_{\sigma\in \Gamma_{L_i}} 
\sigma^{-1}\circ \chi_{\sigma L_i}^{r_{\sigma}}$
for some $r_{\sigma}\in \{0,1\}$
on $I_K$. 
Furthermore, we find 
$\mrm{Nr}_{k/\mbb{Q}_p}\circ \chi_{\pi}
=\prod_{\sigma\in \Gamma_{L_i}} \sigma^{-1}\circ \chi_{\sigma L_i}$
as characters from $I_K$ to $L_i^{\times}$.
Then we have an equality  
$\psi_i^{-1}\cdot ({\mrm{Nr}_{k/\mbb{Q}_p} \circ \chi_{\pi}})
=\prod_{\sigma\in \Gamma_{L_i}} \sigma^{-1}\circ \chi_{\sigma L_i}^{1-r_{\sigma}}$
on $I_K$ and this in particular implies that all the Hodge-Tate weights of 
$L_i(\psi_i^{-1}\cdot ({\mrm{Nr}_{k/\mbb{Q}_p} \circ \chi_{\pi}}))$ are non-negative. 
By \eqref{PsiK} and Lemma \ref{Phi} (5), we obtain 
$\psi_i={\mrm{Nr}_{k/\mbb{Q}_p} \circ \chi_{\pi}}$ on $G_K$.
Thus the Galois group ${G_{Kk_{\pi}}}$ acts on $V_i$ trivial.
This in particular implies that  
$V_p(B)^{G_{Kk_{\pi}}}$ is not zero, which is equivalent to say that 
$B(Kk_{\pi})[p^{\infty}]$ is infinite. 
Therefore, $k_{\pi}$ is not Kummer-faithful.

(3) 
First we consider the case where $\hat{\pi}^2=q$.
Calculating the $p$-adic valuation on both sides,
we have $2\sum_{\sigma\in \Gamma_k} r_{\sigma}=[k:\mbb{Q}_p]$.
On the other hand, taking the norm, we also have 
$\mrm{Nr}_{k/\mbb{Q}_p}(\hat{\pi})^2=q^{[k:\mbb{Q}_p]}$.
Since $k$ is a Galois extension of $\mbb{Q}_p$,
we have $\mrm{Nr}_{k/\mbb{Q}_p}(\hat{\pi})^2=\mrm{Nr}_{k/\mbb{Q}_p}(\pi)^{2\sum_{\sigma\in \Gamma_k} r_{\sigma}}$.
Hence we see that $q^{-1}\mrm{Nr}_{k/\mbb{Q}_p}(\pi)$ is a root of unity, and hence $k_{\pi}$ is not Kummer-faithful
by (2).

In the rest of the proof,
we assume $\hat{\pi}^2\neq q$.
Then $\mbb{Q}(\hat{\pi})$ can not be embedded into $\mbb{R}$.
Moreover, $\mbb{Q}(\hat{\pi})$  must be a CM field
since it is a totally imaginary quadratic extension of 
the totally real number field $\mbb{Q}(\hat{\pi}+q/\hat{\pi})$.
The proof below is based on the method used in the proof of (2) but we need more caferul treatments. 
The notable difference here is that we apply a CM lifting theorem due to 
Chai-Conrad-Oort \cite{CCO14}, which refines the Honda–Tate lifting theorem.
Let $\bar{A}$ be  the   simple abelian variety defined over $\mbb{F}_q$ 
which corresponds to the Weil $q$-number 
$\hat{\pi}$ via Honda-Tate theory.
We put $g=\dim \bar{A}$, 
$D=\mrm{End}_{\mbb{F}_q}(\bar{A})\otimes_{\mbb{Z}} \mbb{Q}$
and $Z=\mbb{Q}(\hat{\pi})$.
Since $\hat{\pi}$ is (a $\mbb{Q}$-conjugation of) the $q$-Frobenius map of $\bar{A}$,  $Z$ is a subfield of $D$.
It is a theorem of Tate (cf.\ Corollary 1.6.2.2 (3) of \cite{CCO14}) that 
$Z$ is a center of $D$, 
$D$ is a division algebra over $Z$ of degree $d^2$ for some integer $d>0$, and 
$2g=d\cdot [Z:\mbb{Q}]$.
Tate moreover showed that 
there exists a maximal subfield $L$ in $D$ of degree $d$ over $Z$
which is a CM field. 
Now we claim that $L=Z$.
For any field $F$, we denote by $\mrm{Br}(F)$  the Brauer group of $F$.
Let $[D]\in \mrm{Br}(Z)$ be the class of $D$. 
We denote by $\mrm{inv}_v\colon \mrm{Br}(Z_v)\overset{\sim }{\rightarrow} \mbb{Q}/\mbb{Z}$  the local invariant map of $\mrm{Br}(Z_v)$
for any finite place $v$ of $Z$. Here, $Z_v$ is the completion of $Z$ at $v$. 
It is known (cf.\ Corollary 1.6.2.2 (3) of \cite{CCO14}) that 
$\mrm{inv}_v([D])=0$ if $v$ is not above $p$ and 
$\mrm{inv}_v([D])=\frac{\mrm{ord}_v (\hat{\pi})f_v}{f} \mod \mbb{Z}$  
if $v$ is above $p$.
By the assumption 1 of (iii) and the fact that $Z$ has no real infinite place, 
we find that $[D]$ is trivial; in fact,
the local invariant maps induces an injection 
$\mrm{Br}(Z)\hookrightarrow \oplus_v \mrm{Br}(Z_v)$ 
where $v$ runs though all finite places of $Z$.
On the other hand, 
the order of $[D]$ in $\mrm{Br}(Z)$  coincides with 
the square root $d$ of $[D:Z]$ by Theorem 1.2.4.4 of \cite{CCO14}.
Thus we have $d=1$ and the claim follows.
Consequently, we obtain the fact that $\bar{A}$ has complex multiplication 
over $\mbb{F}_q$ by the CM field $L=Z$.

Let $\mbb{Q}_q$ be the unramified extension of $\mbb{Q}_p$ of degree $f$.
By Theorem 4.1.1 of \cite{CCO14} (and the constduction of ``$D$" in page 246--247 of {\it loc.\ cit.}), 
there exists a finite totally ramified extension $K/\mbb{Q}_q$ 
and an abelian variety $B$ over $K$ with good reduction  
such that $\bar{A}$ is $\mbb{F}_q$-isogenous to the reduction of $B$.
Moreover, we can take $B$ so that 
$B$ has complex multiplication over $K$ by the same 
choosed CM field as $\bar{A}$;
thus, we may suppose that $B$ has complex multiplication  by $L$.
Let $\prod_i L_i$ denote the decomposition of $L\otimes_{\mbb{Q}} \mbb{Q}_p$
into a finite product of finite extensions of $\mbb{Q}_p$
and let $V_p(B)\simeq \oplus_i V_i$ be the corresponding decomposition of $G_K$-modules.
Each $V_i$ is naturally equipped with a structure of one dimensional 
$L_i$-representation of $G_K$.
By choosing a $\mbb{Q}_p$-algebra embedding 
from $L_i$ into our fixed $\overline{\mbb{Q}}_p$ for each $i$,
we may regard $L_i$ as a subfield of $\overline{\mbb{Q}}_p$.
We denote by $\psi_i\colon G_K\to L_i^{\times}$
the character obtained by the $G_K$-action on $V_i$.
By a similar argument of the proof of (2),
we know that 
the set of the  roots of 
$\mrm{det}(T-\mrm{Frob}_{K}\mid V_{\ell}(\bar{A}))$ 
is $\{\Phi_K(\tau\circ \psi_i)\mid \tau \in \Gamma_{L_i}, i\}$.
Furthermore, $\hat{\pi}$ is also a root of this polynomial 
since $K$ is totally ramified
over $\mbb{Q}_q$. 
Thus we obtain 
\begin{equation}
\label{PsiK2}
\Phi_K( \tau \circ \psi_i) 
= \hat{\pi}
\end{equation}
for some $i$ and some $\mbb{Q}_p$-algebra embedding  $\tau\colon L_i\hookrightarrow \overline{\mbb{Q}}_p$.
Since the left hand side of \eqref{PsiK2} is contained in  $\tau L_i$, 
we have  $\tau L_i\supset \mbb{Q}_p(\hat{\pi})$.
In particular, we have 
$[L_i:\mbb{Q}_p]=n\cdot [\mbb{Q}_p(\hat{\pi}):\mbb{Q}_p]$ for some 
integer $n\ge 1$.
Since $[L_i:\mbb{Q}_p]$ coincides with the degree of some 
irreducible factor of $f(X)$ in $\mbb{Q}_p[X]$,
it follows from the assumption 2 of (iii) that  $n=1$, that is,
\begin{equation}
\label{tauL=k}
\tau L_i= \mbb{Q}_p(\hat{\pi}).
\end{equation}

Note that \eqref{PsiK2} implies 
$
\Phi_K( \tau \circ \psi_i) 
=\Phi_k\left(\prod_{\sigma \in \Gamma_k} \sigma\circ \chi_{\pi}^{r_{\sigma}}\right).
$
Taking $f_{Kk}$-th power of this equality,
we obtain 
$
\Phi_{Kk}((\tau \circ \psi_i)^{f_K}) 
=\Phi_{Kk}\left(\left(\prod_{\sigma \in \Gamma_k} \sigma\circ \chi_{\pi}^{r_{\sigma}}\right)^{f_k}\right).
$
Here we remark that, by \eqref{tauL=k}, we may consider $\tau \circ \psi_i$ as a crystalline character
of $G_K$ with values in $k^{\times}$. 
Combining this with  the assumption that $k$ is a Galois extension of $\mbb{Q}_p$ and the assumption (ii),
it follows from Lemma \ref{Phi} (6) that we have 
$$
(\tau\circ \psi_{i})^{f_K}=\prod_{\sigma \in \Gamma_k} \sigma\circ \chi_{\pi}^{m_{\sigma}}
$$
on $G_{K'}$ for some finite extension $K'$ of $Kk$  and some integers $m_{\sigma}$. 
Thus the character $\tau \circ \psi_i$ restricted to $G_{K'k_{\pi}}$ has values in the set of $f_K$-th roots of unity, 
and hence $\psi_i$ is trivial on $G_{K''k_{\pi}}$
for some finite extension $K''$ of $K'$,
which in particular implies that $B(K''k_{\pi})[p^{\infty}]$ is infinite.
Therefore, we conclude  that $k_{\pi}$ is not Kummer-faithful.
\end{proof}

Theorem \ref{refined:LT} naturally leads us to consider the following question.
\begin{question}
Let $k$ be a Galois extension of $\mbb{Q}_p$. 
Are the following conditions equivalent?
\begin{itemize}
\item[(i)] $k_{\pi}$ is not Kummer-faithful.
\item[(ii)] Either of the following holds:
\begin{itemize}
\item[{\rm (a)}] $q^{-1}\mrm{Nr}_{k/\mbb{Q}_p}(\pi)\in \mu_{p-1}$.
\item[{\rm (b)}] For some $(r_{\sigma})_{\sigma \in \Gamma_k}$
with $r_{\sigma}\in \{0,1\}$, 
it holds $\hat{\pi}:=\prod_{\sigma \in \Gamma_k} \sigma \pi^{r_{\sigma}}\in W(q)$.
\end{itemize}
\end{itemize}
\end{question}
The implications (i) $\Rightarrow$ (ii) 
and (ii-a) $\Rightarrow$ (i)  were shown in  
Theorem \ref{refined:LT} (1) and Proposition \ref{TKF},
respectively.
Thus the remaining problem is to determine whether 
(ii-b) always implies (i). 
At the moment, the author does not have an answer for this problem.
Theorem \ref{refined:LT} (2) gives a partial answers for this;
if $\hat{\pi}$ is the norm $\mrm{Nr}_{k/\mbb{Q}_p}(\pi)$, 
then $\hat{\pi}\in W(q)$ implies (i). 
It is natural to ask whether there actually exists an example in which 
$k_{\pi}$ is not Kummer-faithful,
even though 
$q^{-1}\mrm{Nr}_{k/\mbb{Q}_p}(\pi)\not\in \mu_{p-1}$,
$\hat{\pi}\neq \mrm{Nr}_{k/\mbb{Q}_p}(\pi)$ and  $\hat{\pi}\in W(q)$.
By applying Theorem \ref{refined:LT} (3), we construct such examples in the next section.

\subsection{Examples of non-Kummer-faithful Lubin-Tate extension fields}
\label{nonKE:example}

Here we give an example of a non-Kummer-faithful Lubin-Tate extension field by applying Theorem \ref{refined:LT} (3).
Let $r$ be a divisor of $p-1$. 
Let $F_1,F_2,\dots ,F_r$ be imaginary quadratic fields such that 
$p$ splits completely by principal ideals for each $F_i$ and 
$\mrm{Gal}(F/\mbb{Q})\simeq 
\prod^r_{i=1}\mrm{Gal}(F_i/\mbb{Q})$,
where $F$ is the composite field of $F_1,F_2,\dots ,F_r$.
Let $\mfrak{p}$ be a finite place of $F$ above $p$.
Denote by $\mfrak{p}_i$ the finite place of $F_i$ below $\mfrak{p}$
and take a generator $\omega_i$ of $\mfrak{p}_i$.
Let $\omega_i^c$ be  the complex conjugate of $\omega_i$; 
we have $\omega_i^c=\omega_i^{-1}p$. 
We set 
$$
\pi_0:=\omega_1\cdot \prod^r_{i=2}\omega_i^c \quad 
\mrm{and}\quad \pi:=\pi_0^{1/r}.
$$
We fix an embedding 
$\overline{\mbb{Q}}\hookrightarrow \overline{\mbb{Q}}_p$
with respected to $\mfrak{p}$. 
With this embedding, $\omega_1$ is a uniformizer of $\mbb{Q}_p$
and $\omega_2^c,\dots, \omega_r^c$ are $p$-adic units.
Set $k:=\mbb{Q}_p(\pi)$.  
Then $k$ is a totally ramified  
extension of $\mbb{Q}_p$ of degree $r$
and $\pi$ is its uniformzer.
Since $r$ divides $p-1$, $k$ is a Galois extension of $\mbb{Q}_p$.
The minimal polynomial $f(X)$ of $\pi$ over $\mbb{Q}$
is a divisor of 
$\prod^r_{i=1}\prod_{\sigma_i\in \{1,c\}}
(X^r-\omega_1^{\sigma_1}\cdots \omega_r^{\sigma_r})$
and each $\omega_1^{\sigma_1}\cdots \omega_r^{\sigma_r}$
is of the form 
$p^a\omega_1^{\pm 1}\cdots \omega_r^{\pm 1}$ .
Hence any irreducible factor in $\mbb{Q}_p[X]$ of $f(X)$ 
is a divisor of some $X^r-\omega_1^{\sigma_1}\cdots \omega_r^{\sigma_r}$.
By Theorem \ref{refined:LT} (3),
we obtain that $k_{\pi}$ is not Kummer-faithful.

\begin{proposition}
\label{nonKF2}
Let $\mbb{Q}_p\subset k_0\subset k$ be finite extensions.
Let $\pi_0$ and $\pi$  be uniformizers of $k_0$ and $k$, respectively.
Denote by $k_{0,\pi_0}$ and $k_{\pi}$
the Lubin-Tate extensions of $k_0$ and $k$ associated with $\pi_0$ and $\pi$,
respectively.
Let $f$ be the residue degree of $k/k_0$.
If $\pi_0^{-f}\mrm{Nr}_{k/k_0}(\pi)$ is a root of unity 
and $k_{0,\pi_0}$ is not Kummer-faithful, then $k_{\pi}$ is not Kummer-faithful.
\end{proposition}

\begin{proof}
The result follows from local class field theory; 
if $\pi_0^{-f}\mrm{Nr}_{k/k_0}(\pi)$ is a root of unity,
then a finite extension of $k_{\pi}$ contains $k_{0,\pi_0}$.
\end{proof}

Let $k_0$ be a $p$-adic field and  
$\pi_0$ a uniformizer of $k_0$ so that 
$k_{0,\pi_0}$ is not Kummer-faithful. 
If we choose $k$ and $\pi$ by one of the following manner,
it follows from Proposition \ref{nonKF2} that 
$k_{\pi}$ is not Kummer-faithful.
\begin{itemize}
\item[(i)] Let $\pi$ be a $n$-th root of unity of $\pi_0$
for an integer $n>0$
and set $k:=k_0(\pi)$.
\item[(ii)] 
Let $k$ be a finite unramified extension of $k_0$
and take a uniformizer $\pi$ of $k$ such that 
$\pi_0^{-f}\mrm{Nr}_{k/k_0}(\pi)=1$.
(The existence of such $\pi$ is assured by, for example, \cite[Chapter V, \S.2, Corollary of Proposition 3]{Ser68}.)
\end{itemize}

\end{document}